\documentclass{amsart}

\usepackage{amsfonts,amssymb,amscd,amsmath,enumerate,verbatim}
\usepackage{mathptmx,color,stmaryrd}

\usepackage[nobysame]{amsrefs}

\usepackage[all]{xy}
\SelectTips{cm}{}

\theoremstyle{plain}
\newtheorem*{Theorem}{Theorem}
\newtheorem{theorem}{Theorem}[section]
\newtheorem{proposition}[theorem]{Proposition}
\newtheorem{lemma}[theorem]{Lemma}
\newtheorem{corollary}[theorem]{Corollary}

\theoremstyle{definition}
\newtheorem{definition}[theorem]{Definition}
\newtheorem{example}[theorem]{Example}
\newtheorem{question}[theorem]{Question}

\theoremstyle{remark}
\newtheorem{remark}[theorem]{Remark}
\newtheorem*{claim}{Claim}

\newtheorem{chunk}[theorem]{}

\numberwithin{equation}{theorem}

\newcommand{\bbZ}{\mathbb{Z}}

\newcommand{\ulx}{\underline{x}}

\newcommand{\mcC}{\mathcal{C}}

\newcommand{\mcV}{\mathcal{V}}
\newcommand{\mcX}{\mathcal{X}}
\newcommand{\mcY}{\mathcal{Y}}

\newcommand{\fm}{\mathfrak{m}}
\newcommand{\fp}{\mathfrak{p}}
\newcommand{\fq}{\mathfrak{q}}

\newcommand{\sfa}{\mathsf{a}}

\newcommand{\sfD}{\mathsf{D}}

\newcommand{\QQ}{\mathsf{Q}}

\newcommand{\nil}{\operatorname{\mathsf{nil}}}
\newcommand{\rad}{\operatorname{\mathsf{rad}}}
\newcommand{\proj}{\operatorname{\mathsf{proj}}}

\newcommand{\susp}{\mathsf{\Sigma}}

\newcommand{\dsing}{\mathsf{D}_{\rm sg}}
\newcommand{\dbcat}{\mathsf{D}^{\mathsf b}}

\newcommand{\eps}{\varepsilon}
\newcommand{\vf}{\varphi}
\newcommand{\ges}{\geqslant}

\newcommand{\lra}{\longrightarrow}
\newcommand{\thra}{\twoheadrightarrow}

\newcommand{\xra}{\xrightarrow}

\newcommand{\add}{\operatorname{add}}
\newcommand{\Add}{\operatorname{Add}}
\newcommand{\ann}{{\operatorname{ann}\,}}

\newcommand{\depth}{{\operatorname{depth}\,}}
\renewcommand{\dim}{{\operatorname{dim}\,}}
\newcommand{\End}{\operatorname{End}}
\newcommand{\Ext}{\operatorname{Ext}}
\newcommand{\gldim}{{\operatorname{gldim}\,}}

\renewcommand{\mod}{{\operatorname{mod}\,}}
\newcommand{\Mod}{{\operatorname{Mod}\,}}
\newcommand{\rank}{\operatorname{rank}}
\newcommand{\Reg}{{\operatorname{Reg}\,}}
\newcommand{\Sing}{{\operatorname{Sing}\,}}
\newcommand{\Spec}{{\operatorname{Spec}\,}}

\newcommand{\syz}{{\Omega}\,}
\newcommand{\thick}{\operatorname{thick}}
\newcommand{\Tor}{\operatorname{Tor}}

\newcommand{\sad}[1]{|\syz^{*}_{\Lambda}{#1}|}
\newcommand{\cent}[1]{{#1}^{\!\mathsf c}}
\newcommand{\env}[1]{{#1}^{\!\mathsf e}}
\newcommand{\ndiff}[2]{{\operatorname{N}({#2}/{#1})}}
\newcommand{\opp}[1]{{ {#1}^{\!\mathsf o}}}
\newcommand{\can}[1]{{\mathsf{ca}(#1)}}
\newcommand{\cande}[2]{{\mathsf{ca}^{#1}(#2)}}

\newcommand{\Ker}{\operatorname{Ker}}
\newcommand{\Hom}{\operatorname{Hom}}

\begin{document}

\title[Annihilation and strong generators]{Annihilation of cohomology and\\
strong generation of module categories}

\author{Srikanth B. Iyengar}
\address{Department of Mathematics, University of Nebraska, Lincoln, NE 68588-0130, USA}
\curraddr{Department of Mathematics, University of Utah, Salt Lake City, UT 84112-0090, USA}
\email{iyengar@math.utah.edu}

\author{Ryo Takahashi}
\address{Graduate School of Mathematics, Nagoya University, Furocho, Chikusaku, Nagoya 464-8602, Japan}
\email{takahashi@math.nagoya-u.ac.jp}

\date{\today}


\thanks{Part of this article is based on work supported by the National Science Foundation under Grant  No.\,0932078000, while the authors were in residence at the Mathematical Sciences Research Institute in Berkeley, California, during the 2012--13 Special Year in Commutative Algebra.  SBI was partly supported by NSF grant DMS-1201889 and grant 228007  from the Simons Foundation; RT was partly supported by JSPS Grant-in-Aid for Young Scientists (B) 22740008, JSPS Grant-in-Aid for Scientific Research (C) 25400038 and JSPS Postdoctoral Fellowships for Research Abroad.}

\begin{abstract}
The cohomology annihilator of a noetherian ring that is finitely generated as a module over its center is introduced. Results are established linking the existence of non-trivial cohomology annihilators and the existence of strong generators for the category of finitely generated modules. Exploiting this link, results of Popescu and Roczen, and Wang concerning cohomology annihilators of commutative rings, and also  results of  Aihara and Takahashi, Keller and Van den Bergh, and Rouquier on strong finite generation of the corresponding bounded derived category,  are generalized to cover excellent local rings and also rings essentially of finite type over a field.
\end{abstract}

\maketitle

\setcounter{tocdepth}{1}
\tableofcontents

\section{Introduction}

The central theme of this article is that the two topics that make up its title are intimately related. Inklings of this can be found in the literature, both on annihilators of cohomology, notably work of Popescu and Roczen~\cite{PopescuRoczen90} from 1990, and on generators for module categories that is of more recent vintage; principally the articles of Dao and Takahashi~\cite{DaoTakahashi11}, and Aihara and Takahashi~\cite{AiharaTakahashi11}.  We make precise the close link between the two topics, by introducing and developing appropriate notions and constructions,  and use it to obtain more comprehensive results than are currently available in either one.

To set the stage for describing this relationship we consider a noetherian ring $\Lambda$ that is finitely generated as a module over its center, $\cent\Lambda$. We call such a $\Lambda$ a \emph{noether algebra}. For any non-negative integer $n$, the elements of $\cent\Lambda$ that annihilate $\Ext^{n}_{\Lambda}(M,N)$, for all $M$ and $N$ in $\mod\Lambda$, form an ideal that we denote $\cande n\Lambda$. It is not difficult to see that one gets a tower of ideals $\cdots \subseteq \cande n{\Lambda}\subseteq \cande {n+1}{\Lambda}\subseteq \cdots $, so their union is also an ideal of $\cent\Lambda$ that we denote $\can \Lambda$, and call the \emph{cohomology annihilator} of $\Lambda$. As $\Lambda$ is noetherian there exists an integer $s$ such that $\can{\Lambda}=\cande s{\Lambda}$. 

The questions that drive the development in this paper are the following: How big (in any measure of size, for example, the dimension of the closed subset of $\Spec\cent\Lambda$ it determines) is $\can{\Lambda}$? Does it contain non-zerodivisors? What is the least integer $s$ as above? Not every ring has a non-zero cohomology annihilator ideal. Indeed, consider the \emph{singular locus} of $\Lambda$, that is to say, the subset
\[
\Sing\Lambda:=\{\fp\in\Spec\cent\Lambda\mid \text{$\gldim\Lambda_{\fp}$ is infinite}\}\,.
\]
Here $\gldim$ denotes global dimension. It is easy to check (see Lemma~\ref{le:ca-localize}) that this is contained in the closed subset of $\Spec\cent\Lambda$ defined by $\mcV(\can{\Lambda})$, the set of prime ideals of $\cent\Lambda$ containing $\can{\Lambda}$. This means that when the ideal $\can{\Lambda}$ contains non-nilpotent elements, $\Sing\Lambda$ is contained in a proper closed subset of $\Spec\cent\Lambda$. However, there are even commutative noetherian rings for which this is not the case; the first examples were constructed by Nagata~\cite{Nagata62}; see  Example~\ref{ex:non-open-loci}.

On the other hand, for any $M$ in $\mod\Lambda$ and integer $n\ge 1$ there is an equality
\[
\ann_{\cent\Lambda}\Ext^{n}_{\Lambda}(M,\Omega^{n} M) = \ann_{\cent\Lambda}\Ext^{\geqslant n}_{\Lambda}(M,\mod \Lambda)
\]
where $\Omega^{n} M$ denotes an $n$th syzygy module of $M$ as a $\Lambda$-module; see Lemma~\ref{le:ca-module}. Observe that $\cande n{\Lambda}$ is the intersection of the ideals on the right, as $M$ varies over $\mod \Lambda$. This suggests the following definition: A finitely generated $\Lambda$-module $G$ is a \emph{strong generator} for $\mod\Lambda$ if there exist integers $s$ and $n$ such that for each $M\in\mod\Lambda$, there is $\Lambda$-module $W$ and a filtration
\[
\{0\}=Z_{0} \subseteq Z_{1}\subseteq \cdots \subseteq Z_{n} = Z\quad\text{where $Z=W \bigoplus \Omega^{s}M$,}
\]
with $Z_{i+1}/Z_{i}$ is in $\mathrm{add}\, G$, for each $i$. We require also that $G$ contains $\Lambda$ as a direct summand, so that $G$ is a generator in the usual sense of the word.

This definition should be compared with that of a generator of a triangulated category introduced by Bondal and Van den Bergh~\cite{BondalVandenBergh03}. Unlike for triangulated categories, there are various possible notions of ``generation'' for module categories, stemming from the fact that, in a module category, kernels and co-kernels are not interchangeable; some of this is clarified in Section~\ref{se:generators}. 
The following result, extracted from Theorem~\ref{th:generator-sing}, links strong generation and existence of cohomology annihilators.

\begin{Theorem}
Suppose $d=\sup\{\gldim\Lambda_{\fp}\mid \fp\not\in\Sing\Lambda\}$ is finite. If $\mod\Lambda$ has a strong generator with parameter $s$ as above, then 
\[
\mcV(\can{\Lambda})  = \mcV(\cande{s+d+1}{\Lambda})= \Sing\Lambda\,.
\]
In particular, $\Sing\Lambda$ is a closed subset of $\Spec\cent\Lambda$. 
\end{Theorem}

One can also describe the cohomology annihilator ideal, at least up to radical, in terms of the generator of $\mod\Lambda$; see Theorem~\ref{th:generator-sing}. Going in the other direction, we prove:

\begin{Theorem}
When $R$ is a commutative noetherian ring of finite Krull dimension and there exists an integer $s$ such that $\cande s{R/\fp}\ne 0$ for each prime ideal $\fp$ in $R$, then $\mod R$ has a strong generator.
\end{Theorem}

This is contained in Theorems~\ref{th:ca-generator} and \ref{th:ca-generator2}. These results shift the focus to finding non-zero cohomology annihilators, and one of the main tools for this is the noether different introduced by Auslander and Goldman~\cite{AuslanderGoldman60}, under the name ``homological different''.  This is explained in Section~\ref{se:the-noether-different}. Building on these results we prove

\begin{Theorem} 
When $R$ is a commutative  ring that is either a finitely generated algebra over a field or an equicharacteristic excellent local ring, then $\mod R$ has a strong generator and  $\mcV(\can R)  = \mcV(\cande {2d+1} R) = \Sing R$, where $d=\dim R$.
\end{Theorem}

This statement is contained in Theorems~\ref{th:complete-equicharacteristic} and \ref{th:eft}. We tackle the case when $R$ is an excellent local ring by passage to its completion, and the argument illustrates well the flexibility afforded by considering cohomology annihilator ideals, rather than focusing on ideals defining the singular locus.

The identification of the singular locus with the closed subset defined by the cohomology annihilator  is related to results of 
Wang~\cites{Wang94,Wang99},  that in turn extend  work of Dieterich~\cite{Dieterich87},  Popescu and Roczen \cite{PopescuRoczen90}, and Yoshino \cites{Yoshino87, Yoshino90} stemming from Brauer-Thrall conjectures for maximal Cohen-Macaulay modules over Cohen-Macaulay rings; see the comments preceding Theorems~\ref{th:complete-equicharacteristic} and \ref{th:eft}.

On the other hand, the part of the statement above dealing with generators extends results of Dao and Takahashi \cite{DaoTakahashi11},  who proved it for complete local rings, assuming that the coefficient field is perfect.  As is explained in Section~\ref{se:dcat}, any strong generator for $\mod R$ gives one for its bounded derived category so one obtains:

\begin{Theorem}
When $R$ is commutative ring that is either  essentially of finite type over a field or an equicharacteristic excellent local ring,  $\dbcat(R)$ is strongly finitely generated. \qed
\end{Theorem}

This is contained in Corollary~\ref{co:dcat-generators}. It generalizes work of Aihara and Takahashi~\cite{AiharaTakahashi11} and Rouquier \cites{Rouquier08a,Rouquier08aa}---see also Keller and Van den Bergh \cite{KellerVandenBergh08}---on the existence of generators for bounded derived categories.

As mentioned before,  the noether different is one of our main tools for finding cohomology annihilators. In Section~\ref{se:ascent-descent}, we develop a different approach, based on tracking the ascent and descent of property that the cohomology annihilator ideal contains a non-zerodivisor, between a commutative ring $A$ and any $A$-algebra $\Lambda$ that is finitely generated as an $A$-module. The following result, contained in  Corollary~\ref{co:ascent-descent-nzd}, is paradigmatic.

\begin{Theorem}
When $M$ is a torsion-free finitely generated $A$-module with rank, for any integer $n$ the ideal $\cande nA$ has a non-zero divisor if and only if so does $\cande n{\End_{A}(M)}$.
\end{Theorem}

A direct corollary is that if a domain admits a noncommutative resolution, in the sense of  \cite{DaoIyamaTakahashiVial12}, then its cohomology annihilator ideal is non-zero; see Remark~\ref{re:ncr}.  Results such as these highlight the benefit of considering cohomology annihilators for (not necessarily commutative) noether algebras even if one is interested only in  commutative rings.

\subsection*{Acknowledgements}
The authors are grateful to Ragnar-Olaf Buchweitz for various conversations and lectures. In particular, he pointed out to us that the study of annihilation of cohomology arose in number theory, in the work of Weil \cite{Weil43}, and was taken up  by Kawada \cites{Kawada51,Kawada52}, Moriya \cite{Moriya53} and Kuniyoshi \cite{Kuniyoshi58}. Thanks are due also to various referees for suggestions for improving the presentation.  The first author also thanks Lucho Avramov for  discussions on the material in this paper. The second author expresses his deep gratitude to David Eisenbud and Shiro Goto for providing opportunities to present this work in their seminars at the University of California, Berkeley, and at Meiji University, respectively. 

\section{Cohomology annihilators}
\label{se:cohomology-annihilators}
In this section, $\Lambda$ will be a \emph{noether} algebra, that is to say,  $\Lambda$ is a noetherian ring that is finitely generated when viewed as a module over its center, that we denote $\cent\Lambda$. Then the ring $\cent\Lambda$ is also noetherian. We write $\Mod \Lambda$ for the category of $\Lambda$-modules; its subcategory consisting of finitely generated modules is denoted $\mod \Lambda$. Our convention is that $\Lambda$ acts on a module from the right; thus left modules will be viewed as modules over the opposite ring to $\Lambda$, denoted $\opp\Lambda$.  Subcategories are assumed  full and closed under isomorphisms.

A $\Lambda$-module is finitely generated if and only if it has that property when viewed as a module over $\cent \Lambda$. It follows then that $\Ext^{n}_{\Lambda}(M,N)$ is a finitely generated $\cent\Lambda$-module for any $M,N\in \mod\Lambda$ and integer $n$. This remark will be used without further comment.

\begin{definition}
\label{de:cohomology-annihilator}
For each non-negative integer $n$, we consider the following ideal of $\cent\Lambda$:
\[
\cande n{\Lambda} := \ann_{\cent\Lambda}\Ext^{\ges n}_{\Lambda}(\mod \Lambda,\mod \Lambda)\,,
\]
In words, this ideal consists of elements $a\in \cent\Lambda$ such that $a\Ext^{i}_{\Lambda}(M,N)=0$ for all $M,N$ in $\mod \Lambda$ and integers $i\ge n$. Note that $\cande n{\Lambda}\subseteq \cande{n+1}{\Lambda}$. The \emph{cohomology annihilator} of $\Lambda$ is the union of these ideals:
\[
\can \Lambda:= \bigcup_{n\ges 0}\cande n{\Lambda}\,.
\]
We also call an element of $\can\Lambda$  a cohomology annihilator; this should cause no confusion, for the context should make it clear whether the ideal or an element is intended. Observe that, as the ring $\cent\Lambda$ is noetherian, there is some integer $s$ such that $\can\Lambda=\cande s\Lambda$. In fact, more is true and is recorded in Remark~\ref{re:ext-and-syzygy}. 
\end{definition}

\subsection*{Syzygy modules}
Let $M$ be a $\Lambda$-module. We write $\syz_{\Lambda}M$ for the kernel of any surjective $\Lambda$-linear map $P\thra M$, where 
$P$ is a projective $\Lambda$-module; when $M$ is finitely generated, $P$ can be chosen to be finitely generated, and \emph{we will tacitly assume that this is the case}. By Schanuel's lemma, $P$ depends only on $M$, up to projective summands. For any integer $n\ge 1$, we set $\syz_{\Lambda}^nM:=\syz_{\Lambda}(\syz^{n-1}M)$, and call it an $n$th syzygy module of $M$; when $M$ is finitely generated, so are the syzygy modules.  We omit $\Lambda$ from the notation, when the ring in question is clear from the context.

The following well-known observations will be used repeatedly in the sequel.

\begin{remark}
\label{re:syzygies}
Let $0\to L\to M\to N\to 0$ be an exact sequence of $\Lambda$-modules. For each non-negative integer $n\ge 1$, there is an induced exact sequence 
\[
0\lra \syz^{n} L \lra \syz^{n} M\lra \syz^{n} N\lra 0
\]
for some choice of syzygy modules for $L$, $M$, and $N$: It suffices to verify this for $n=1$, and then it is immediate from the Horseshoe Lemma.
\end{remark}

\begin{remark}
\label{re:ext-and-syzygy}
Fix $M,N$ in $\Mod \Lambda$ and a sequence $0\to\syz M\to P\to M\to 0$ defining $\syz M$. The induced surjective map 
\[
\Hom_{\Lambda}(\syz_{\Lambda}M,N)\to \Ext^{1}_{\Lambda}(M,N)
\]
is  $\cent\Lambda$-linear as are the isomorphisms 
\[
\Ext^{n+1}_{\Lambda}(M,N)\cong \Ext^{n}_{\Lambda}(\syz_{\Lambda}M,N) \quad\text{for $n\ge 1$}.
\]
It follows that $\cande n\Lambda = \ann_{\cent\Lambda}\Ext^{n}_{\Lambda}(\mod \Lambda,\mod\Lambda)$ for each $n$. Summing up, there exists some integer $s$ such that
\[
\can\Lambda = \ann_{\cent\Lambda}\Ext^{s}(\mod \Lambda,\mod\Lambda)
\]
The least such integer $s$ is evidently an invariant of $\Lambda$. Here is a lower bound on it.

\begin{proposition}
\label{pr:ca-zero}
Let $\Lambda$ be a noether algebra.  If the Jacboson radical of $\cent \Lambda$ contains a $\Lambda$-regular sequence of length $d$, then $\cande d{\Lambda}=0$.
\end{proposition}

\begin{proof}
Let $\ulx=x_{1},\dots,x_{d}$ be a $\Lambda$-regular sequence in the Jacobson radical of $\cent\Lambda$. The for each integer $n\ge 1$, the
sequence $\ulx^{n}:=x_{1}^{n},\dots,x_{d}^{n}$ is also $\Lambda$-regular; see \cite{Matsumura80}*{Theorem 26}. Given this, it is not hard to verify that $\Ext^{d}_{\Lambda}(\Lambda/(\ulx^{n}),\Lambda)$ is isomorphic to $\Lambda/(\ulx^{n})$ as $\cent\Lambda$-modules. Thus $(\ulx^{n})\cent\Lambda$ annihilates this Ext-module and hence
\[
\cande d{\Lambda} \subseteq \bigcap_{n\ges 1}(\ulx^{n})\cent \Lambda = (0)\,.
\]
The inclusion is by Krull's intersection theorem~\cite{Nagata62}*{Theorem 4.2}; this is where one needs the hypothesis that $\ulx$ is in the Jacobson radical of $\cent\Lambda$.
\end{proof}
\end{remark}

As we shall soon see  there may  be no non-zero cohomology annihilators, even for ``reasonably nice'' rings. But first we describe some examples where we can readily identify some interesting annihilators.

\begin{example}
\label{ex:regular-rings}
Recall that the ring $\Lambda$ has \emph{global dimension $\le d$}, for some integer $d$, if
\[
 \Ext_\Lambda^{d+1}(\mod\Lambda,\mod\Lambda)=0\,.
 \]
This is the case if and only if $\cande {d+1}\Lambda = \cent \Lambda$.
\end{example}

In what follows, when discussing the commutative ring case, we use $R$ instead of $\Lambda$. In this context,  \emph{local}, means also noetherian.

\begin{example}
Let $R$ be a local ring and let $\fm$ denote its maximal ideal. The socle of $R$ annihilates $\Ext_R^{\ges 1}(\mod R,\mod R)$. In particular, if $R$ is artinian and $l$ is its Loewy length:
\[
l=\inf\{n\ge 0\mid \fm^{n} = 0\}
\]
there is an inclusion $\fm^{l-1}\subseteq\cande 1R$.

Indeed, let $M$ be a finitely generated $R$-module and $F \thra M$ its projective cover; its kernel, $\syz M$, is thus contained in $\fm F$. In particular, the socle of $R$ annihilates $\syz M$ and hence also $\Hom_R(\syz M,N)$. Since there is a surjection $\Hom_R(\syz M,N)\thra\Ext_R^1(M,N)$, the assertion follows; see Remark~\ref{re:ext-and-syzygy}.
\end{example}

\begin{example}
\label{ex:us-hypersurfaces}
Let $k$ be a field and $k\llbracket \ulx\rrbracket$ the formal power series ring in commuting indeterminates  $\ulx = x_{0},\dots, x_{d}$ over $k$. Let $f$ be an element in $k\llbracket \ulx\rrbracket$ and set $R=k\llbracket \ulx\rrbracket/(f)$. Then
\[
 \cande {d+1}R \supseteq \big(\frac{\partial f}{\partial x_{0}},\dots,\frac{\partial f}{\partial x_{d}}\big)\,.
\]
This computation is due to Dieterich~\cite{Dieterich87}*{Proposition~18}.
\end{example}

\begin{example}
\label{ex:not-uap}
Let $k$ be a field and $R=k\llbracket x,y\rrbracket/(x^2)$. Then $\can R =\cande 2R = (x)$.

Indeed, the $R$-module $R/(x)$ has a free resolution
\[
\cdots \xra{\ x\ } R \xra{\ x\ } R \xra{\ x\ } R \lra 0\,,
\]
so $\Ext_R^i(R/(x),R/(x))\cong R/(x)$ for all $i\ge0$, so $\can R\subseteq (x)$. It remains to check verify that there is an inclusion $(x)\subseteq 
\cande 2R$.  When the characteristic of $k$ is not two, this follows from Example~\ref{ex:us-hypersurfaces}. In the remaining case, one can verify it by a direct calculation, using the classification of maximal Cohen-Macaulay $R$-modules.
\end{example}

We record a basic obstruction to the existence of cohomology annihilators.

\subsection*{Regular and singular loci}
Let $A$ be a commutative noetherian ring. We say that $\Lambda$ is a \emph{noether $A$-algebra} if it is an $A$-algebra that  is finitely generated as an $A$-module; in particular, $\Lambda$ is noetherian. In this situation, $\Lambda$ is also finitely generated over its center,  since that the action of $A$ on $\Lambda$ factors through $\cent\Lambda$, that is to say, $\Lambda$ is a noether algebra.

\begin{lemma}
\label{le:gldim-descent}
If  $\Lambda$ is a noether $A$-algebra and free as an $A$-module, then $\gldim A \leq \gldim \Lambda$. 
\end{lemma}

\begin{proof}
This is a consequence of the fact that, since $\Lambda$ is  free $A$-module of finite rank, for any finitely generated $A$-modules $M,N$ and integer $n$, there are isomorphisms
\[
\Ext^{n}_{A}(M,N)\otimes_{A}\Lambda \cong \Ext^{n}_{\Lambda}(M\otimes_{A}\Lambda, N\otimes_{A}\Lambda)\,.
\]
Note that $M\otimes_{A}\Lambda$ and $N\otimes_{A}\Lambda$ are finitely generated $\Lambda$-modules.
\end{proof}

As usual, we write $\Spec R$ for the collection of prime ideals in a commutative ring $R$, with the Zariski topology. Thus, the closed sets in the topology are the subsets
\[
\mcV(I) :=\{\fp\in\Spec R\mid \fp\supseteq I\}\,,
\]
for ideals $I\subseteq R$. Extending a notion from commutative algebra, we introduce the \emph{regular locus} and the \emph{singular locus} of a noether algebra as the following subsets of $\Spec \cent\Lambda$:
\[
\Reg \Lambda = \{\fp\in\Spec\cent\Lambda\mid \gldim \Lambda_{\fp} <\infty\} \quad\text{and}\quad
\Sing \Lambda = \Spec\cent\Lambda \setminus \Reg \Lambda\,.
\]
When $\Lambda$ is free as a $\cent\Lambda$-module, it follows from Lemma~\ref{le:gldim-descent} that $\Reg \Lambda\subseteq \Reg \cent\Lambda$.

\begin{lemma}
\label{le:ca-localize}
Let $\Lambda$ be a noether algebra. 
\begin{enumerate}[{\quad\rm(1)}]
\item For any multiplicatively closed subset $U\subset \cent\Lambda$ there are inclusions
\[
U^{-1} \cande n\Lambda\subseteq \cande n{U^{-1}\Lambda}\quad\text{and}\quad U^{-1} \can \Lambda \subseteq \can {U^{-1}\Lambda}\,.
\]
\item
There is an inclusion $\Sing\Lambda\subseteq \mcV(\can \Lambda)$.
\item
If $\can\Lambda$ is not nilpotent, $\Reg\Lambda$ contains a non-empty open subset of $\Spec\cent\Lambda$.
\end{enumerate}
\end{lemma}

\begin{proof}
The crux of the proof of (1) is that each finitely generated module over $U^{-1}\Lambda$ has the form $U^{-1}M$, for some $M\in\mod\Lambda$. It remains to note that  there are isomorphisms
\[
U^{-1}\Ext_{\Lambda}^n(M,N) \cong \Ext^n_{U^{-1}\Lambda}(U^{-1}M,U^{-1}N)\,,
\]
for any $M,N\in \mod\Lambda$ and integer $n$, and that $\cent{(U^{-1}\Lambda)}=U^{-1}\cent\Lambda$.

\medskip

(2) For $\fp\in\Spec \cent\Lambda$ with $\fp\not \supseteq \can\Lambda$, it follows from (1) that $\can{\Lambda_{\fp}}=(1)$, and hence that $\gldim \Lambda_{\fp}$ is finite. 

\medskip

(3) This follows from (2) because $\mcV(\can\Lambda)\ne \Spec\cent\Lambda$ when $\can\Lambda$ contains non-nilpotent elements.
\end{proof}

Part (2) of the preceding result raises the question: For which rings is there an equality $\Sing\Lambda = \mcV(\can\Lambda)$?
See Theorem~\ref{th:generator-sing} for a partial answer.  Such an equality implies, in particular, that the singular locus of $\Lambda$ is a closed subset of $\Spec\cent\Lambda$, and this is not always the case, even for commutative rings.

\begin{example}
\label{ex:non-open-loci}
The first systematic investigation of rings with non-closed singular loci is due to Nagata~\cite{Nagata59}*{\S\S~4,5}. A particularly simple procedure for constructing such examples was discovered by Hochster~\cite{Hochster73}*{Example 1}, who used it to describe one-dimensional domain $R$ with countably infinitely many prime ideals of height one such that $\Reg R=\{(0)\}$, and the intersection of any infinite set of maximal ideals of $R$ is $(0)$. Thus $\Sing R$ is not contained in any closed set.  Ferrand and Raynaud~\cite{FerrandRaynaud70}*{Proposition~3.5} have constructed a three-dimensional local domain containing $\mathbb{C}$ whose singular locus is not closed.
\end{example}

\subsection*{Annihilators of $\Ext^{1}$}
Let $\Lambda$ be an associative ring and let $M,N$ be $\Lambda$-modules. We record a few simple observations on the annihilators of $\Ext^{1}_{\Lambda}(M,N)$, for later use.  The one below has been made before; see \cite{HerzogPopescu97}*{Lemma 2.2} and \cite{Takahashi10}*{Lemma 2.1}.
 
\begin{remark}
\label{re:ext2}
Let $M$ be a $\Lambda$-module. If $a\in\cent\Lambda$ annihilates $\Ext^{1}_{\Lambda}(M,\syz_{\Lambda}M)$, then there is an exact sequence of $\Lambda$-modules
\[
0\lra (0\,\colon_{M}\,a)\lra M\bigoplus \syz_{\Lambda}M \lra \syz_{\Lambda}(M/aM)\lra 0\,.
\]
Indeed, consider the commutative diagram with exact rows:
\[
\xymatrixrowsep{1.5pc}
\xymatrix{
0 \ar@{->}[r] & \syz M \ar@{->}[r] \ar@{=}[d] & N \ar@{->}[r] \ar@{->}[d] & M \ar@{->}[r]\ar@{->}[d]^{a} & 0 \\
0 \ar@{->}[r] &\syz M \ar@{->}[r] & P \ar@{->}[r] & M \ar@{->}[r] & 0}
\]
where the lower one, with $P$ is projective, defines $\syz M$ and the upper one is obtained from it by a pull-back along the map $M\xra{a}M$. Since $a$ annihilates $\Ext_\Lambda^1(M,\syz_\Lambda M)$, the upper sequence splits so that $N\cong M\oplus\syz_\Lambda M$. It remains to invoke the Snake Lemma.
\end{remark}

\begin{remark}
\label{re:ext1}
Fix  $\xi\in \Ext^{1}_{\Lambda}(M,N)$ represented by the extension $0\to N\to E\xra{\eps} M\to 0$. It is easy to verify that for $a\in\cent\Lambda$, one has $a\xi=0$ if and only if the homothety defined by $a$ on $M$ factors through $\eps$; that is to say, there is a commutative diagram of $\Lambda$-modules:
\[
\xymatrixcolsep{1pc}
\xymatrixrowsep{1pc}
\xymatrix{
M \ar@{->}[rr]^{a} && M \\
& E \ar@{<-}[ul] \ar@{->}[ur]_{\eps}
}
\]
\end{remark}

Here is a consequence of this observation; it can also be deduced from Remark~\ref{re:ext2}. This result will be used repeatedly in what follows. Observe that, in view of the stated equality, the ideal on the right is independent of the choice of the syzygy module, $\syz^{n}_{\Lambda} M$. 

\begin{lemma}
\label{le:ca-module}
For any $\Lambda$-module $M$ and choice of $n$th syzygy module $\syz^{n}_{\Lambda} M$, for an integer $n\ge 1$, there is an equality
\[
\ann_{\cent\Lambda}\Ext^{\ges n}_{\Lambda}(M,\Mod \Lambda) = \ann_{\cent\Lambda}\Ext^{n}_{\Lambda}(M,\syz^{n}_{\Lambda} M) \,.
\]
\end{lemma}

\begin{proof}
Evidently, the ideal on the left is contained in the one on the right, so it suffices to verify the reverse containment.  Replacing $M$ by $\syz^{n-1}M$ we may assume that $n=1$. Let $0\to \syz M\to P\to M\to 0$ be the exact sequence of $\Lambda$-modules, with $P$ projective, defining $\syz M$, and let $a$ be an element in $\cent\Lambda$ that annihilates this, when viewed as a class in $\Ext^{1}_{\Lambda}(M,\syz M)$. From Remark~\ref{re:ext1} one then gets the commutative diagram on the left:
\[
\begin{gathered}
\xymatrixcolsep{1pc}
\xymatrixrowsep{1pc}
\xymatrix{
M \ar@{->}[rr]^{a} && M \\
& P \ar@{<-}[ul] \ar@{->}[ur]
}
\end{gathered}
\quad\quad
\begin{gathered}
\xymatrixcolsep{1pc}
\xymatrixrowsep{1pc}
\xymatrix{
\Ext^{i}_{\Lambda}(M,N) \ar@{<-}[rr]^{a} && \Ext^{i}_{\Lambda}(M,N) \\
& \Ext^{i}_{\Lambda}(P,N) \ar@{->}[ul] \ar@{<-}[ur]
}
\end{gathered}
\]
The one on the right is obtained from the one on the left by applying $\Ext^{i}(-,N)$, where $N$ is any $\Lambda$-module and $i$ any integer. It remains to note that $\Ext^{i}_{\Lambda}(P,N)=0$ when $i\ge 1$.
\end{proof}

\section{The noether different}
\label{se:the-noether-different}

In this section we explain how certain ideas introduced by Noether~\cite{Noether50}, and developed by Auslander and Goldman~\cite{AuslanderGoldman60}, and also Scheja and Storch~\cite{SchejaStorch73}, can be used to find cohomology annihilators, especially those that are non-zerodivisors. The results presented are inspired by, and extend to not necessarily commutative rings, those of Wang~\cites{Wang94,Wang99}; see also \cites{PopescuRoczen90,Yoshino90}. The novelty, if any, is that some arguments are simpler, and we contend more transparent; confer the proofs of Lemma~\ref{le:ndiff-annihilates} below with that of \cite{Wang94}*{Proposition 5.9}. 

Let $A$ be a commutative ring and $\Lambda$ an $A$-algebra. The opposite ring  of $\Lambda$, denoted $\Lambda^{\!\mathsf o}$, is also an $A$-algebra,  and so is its enveloping algebra $\env \Lambda := \opp \Lambda\otimes_{A} \Lambda$. Then $\Lambda$ is a right module over $\env\Lambda$, with action given by $\lambda\cdot (x\otimes y)=x\lambda y$. The map
\begin{equation}
\label{eq:mu}
\mu\colon \env \Lambda \lra \Lambda \quad \text{defined by $\mu(x\otimes y)=xy$}
\end{equation}
is a surjective homomorphism of right $\env \Lambda$-modules. Recall that $\Hom_{\env\Lambda}(\Lambda,\Lambda)$ is the center, $\cent\Lambda$, of $\Lambda$. The image of the induced map
\[
\Hom_{\env\Lambda}(\Lambda,\mu)\colon \Hom_{\env \Lambda}(\Lambda,\env\Lambda)\lra \Hom_{\env \Lambda}(\Lambda,\Lambda) = \cent\Lambda
\]
is the \textit{noether different} of $\Lambda$ over $A$; we denote it $\ndiff A\Lambda$. It is an ideal in $\cent\Lambda$, and can be identified with $\mu(\ann_{\env \Lambda}\Ker\mu)$. 

\begin{remark}
\label{rem:ndiff}
Fix an element $x\in \cent\Lambda$. The homothety map $\Lambda\xra{\ x \ }\Lambda$ is then $\env\Lambda$-linear and it is immediate from the definition that this map admits a $\env\Lambda$-linear factorization 
\[
\Lambda\xra{\ \eta_{x} \ } \env\Lambda\xra{\mu}\Lambda\,.
\]
if and only if $x$ is in $\ndiff A\Lambda$. In this case, for each $\Lambda$-module $M$ one gets, on applying $M\otimes_{\Lambda}-$, a factorization as $\Lambda$-modules:
\begin{equation}
\label{eq:ndiff-lifting}
\xymatrixcolsep{2.5pc}
\xymatrix{
M \ar@/^1.5pc/[rr]^{x} \ar@{->}[r]_-{M\otimes\eta_{x}} & M \otimes_{A} \Lambda \ar@{->}[r]_{M \otimes \mu} & M 
}
\end{equation}
\end{remark}

When $\Lambda$ is commutative, the following result is~\cite{Wang94}*{Proposition 5.9}; the proof we offer is also different from the one in \emph{op.~cit.}

\begin{lemma}
\label{le:ndiff-annihilates}
For any $\Lambda$-modules $M$ and $N$, there is an inclusion
\[
\ndiff A\Lambda \cdot \ann_{A}\Ext^{1}_{A}(M,N) \subseteq \ann_{\cent\Lambda} \Ext^{1}_{\Lambda}(M,N)\,.
\]
\end{lemma}

\begin{proof}
Fix an extension $0\to N\to L\xra{\ g\ }M\to 0$ of $\Lambda$-modules and an element $a$ in $A$ that annihilates it, when viewed as an extension of $A$-modules. Thus, there exists an $A$-linear map $h\colon M\to L$ such that $gh = a$; see Remark~\ref{re:ext1}. The desired statement is that, given an element $x$ in $\ndiff A\Lambda$, there is a $\Lambda$-linear map $M\to L$ whose composition with $g$ is the map $M\xra{\ xa\ } M$. Such a map is furnished by composing maps along the unique path from $M$ to $L$ in the commutative diagram of $\Lambda$-modules:
\[
\xymatrixcolsep{2.5pc}
\xymatrix{
	&M \ar@{<-}[d]^{M\otimes \mu}\ar@{->}[drrr]^{a}& \\
M \ar@{->}[ur]^{x} \ar@{->}[r]_-{M\otimes \eta_{x}} & M\otimes_{A}\Lambda \ar@{->}[r]_{h\otimes \Lambda} 
	& L \otimes_{A} \Lambda \ar@{->}[r]_-{L \otimes \mu} & L \ar@{->}[r]_{g}& M 
}
\]
The triangle on the left is from \eqref{eq:ndiff-lifting}, and is commutative by construction; that of the one on the right is a direct verification.
\end{proof}

In order to proceed we need the following result that elaborates the connection between degree-shifting of Ext and syzygies; confer Remark~\ref{re:ext-and-syzygy}. Henceforth let $\Lambda$ be a noether $A$-algebra so that finitely generated $\Lambda$-modules are also finitely generated as $A$-modules.

\begin{lemma}
\label{le:ext-and-syzygy}
Let $\Lambda$ be a noether $A$-algebra and set $I = \ann_{A}\Ext^{1}_{A}(\Lambda,\syz_{A}\Lambda)$. For each integer $n\ge 1$, and for $0\le i \le n-1$, there is an inclusion
\[
I^{i} \ann_{A}\Ext^{n}_{A}(\mod \Lambda,\mod\Lambda)\subseteq \ann_{A}\Ext^{n-i}_{A}(\syz_{\Lambda}^{i}(\mod \Lambda),\mod\Lambda)\,.
\]
\end{lemma}

\begin{proof}
The proof is an induction on $i$; the base case $i=0$ is a tautology. Assume that the desired inclusion holds for some integer $i$ with $0\le i< n-1$.  Fix $M$ and $N$ in $\mod\Lambda$. The exact sequence of $\Lambda$-modules 
\[
0 \to \syz_{\Lambda }^{i+1}M \to P \to \syz_{\Lambda }^{i}M \to 0
\]
with $P$ a finitely generated projective,  induces an exact sequence of $A$-modules
\[
\Ext_A^{n-i-1}(P,N)\to\Ext_A^{n-i-1}(\syz_{\Lambda}^{i+1}M,N)\to\Ext_A^{n-i}(\syz_{\Lambda}^{i}M,N)
\]
Write $J$ for the annihilator of $\Ext^{n}_{A}(\mod\Lambda,\mod\Lambda)$ as an $A$-module. In the sequence above, the module on the right is annihilated by $I^{i}J$ by the induction hypothesis, while the one on the left is annihilated by $I$; this follows from Lemma~\ref{le:ca-module}, as $n-i-1\ge1$ and $P$ is a direct summand of a free $\Lambda$-module. Thus, $I^{i+1}J$ annihilates the module in the middle. This completes the induction step.
\end{proof}

The result below is our main tool for finding cohomology annihilators. Its proof in fact shows that  $I^{d} \cdot \ndiff A{\Lambda}$ annihilates $\Ext_{\Lambda}^{n}(M,N)$ for $n > d$, where $d=\gldim A$, for \emph{all} $\Lambda$-modules $M$ and $N$, and not only for the finitely generated ones.

\begin{proposition}
\label{pr:noether-normalization}
Let $A$ be commutative noetherian ring, $\Lambda$ a noether $A$-algebra, and set $I=\ann_{A}\Ext_A^{1}(\Lambda,\syz_{A}\Lambda)$.
If $d:=\gldim A$ is finite, then 
\[
I^{d}\cdot \ndiff A{\Lambda}\subseteq \cande {n}{\Lambda} \quad\text{for $n\ge d+1$.}
\]
\end{proposition}

\begin{proof}
The hypothesis on $A$ is that $\Ext^{d+1}_{A}(\mod A,\mod A)=0$. Thus, for any $\Lambda$-modules $M$ and $N$, Lemma~\ref{le:ext-and-syzygy} (applied with $n=d+1$ and $i=d$) yields an inclusion
\[
I^{d}\subseteq \ann_{A} \Ext_{A}^{1}(\syz^{d}_{\Lambda}M,N)\,.
\]
Lemma~\ref{le:ndiff-annihilates} then implies  $ I^d\ndiff A{\Lambda}$ annihilates $\Ext_{\Lambda}^{1}(\syz^{d}_{\Lambda}M,N)\cong \Ext_{\Lambda}^{d+1}(M,N)$. This justifies the desired inclusion for $n=d+1$; it remains to recall Remark~\ref{re:ext-and-syzygy}.
\end{proof}

The preceding result is effective only when the noether different is non-zero, whence the import of the next one. An $A$-algebra $\Lambda$ is \emph{separable} is $\ndiff A{\Lambda}=\cent\Lambda$; see \cite{AuslanderGoldman60}.

\begin{lemma}
\label{le:ndiff-separable}
Let $A$ be a commutative ring and $Q$ its ring of fractions. If $\Lambda$ is a noether $A$-algebra such that the $Q$-algebra $Q\otimes_{A}\Lambda$ is separable, then $\ndiff A\Lambda$ contains a non-zerodivisor of $\cent\Lambda$.
\end{lemma}

\begin{proof}
The separability hypothesis and \cite{AuslanderGoldman60}*{Proposition 1.1} imply 
\[
\ndiff{Q}{(Q\otimes_{A}\Lambda)}=\cent {(Q\otimes_{A}\Lambda)}\,.
\]
As $\ndiff{Q}{(Q\otimes_{A}\Lambda)} = Q\otimes_{A}{\ndiff A{\Lambda}}$ by \cite{AuslanderGoldman60}*{Proposition 4.4}, the desired result follows.
\end{proof}

\subsection*{Separable noether normalization}
We say that a ring $\Lambda$ has a \emph{separable noether normalization} if there exists a subring $A$ of $\cent \Lambda$ such that the following conditions holds:
\begin{enumerate}[{\quad\rm(i)}]
\item $A$ is noetherian and of finite global dimension;
\item $\Lambda$ is finitely generated as an $A$-module, thus a noether $A$-algebra;
\item $Q\otimes_{A}\Lambda$ is separable over $Q$; here $Q$ is the ring of fractions of $A$. 
\end{enumerate}

The next result  plays a crucial role in Section~\ref{se:commutative} where it is used to prove that the module categories of certain rings have generators, in the sense explained in next section. That in turn allows one to identify the full cohomology annihilator ideal, at least up to radical.

\begin{theorem}
\label{th:separable-noether-normalization}
If $\Lambda$ admits a separable noether normalization, then for $d=\dim\cent\Lambda$ the ideal $\cande {d+1}\Lambda$ of $\cent\Lambda$ is non-zero.
\end{theorem}

\begin{proof}
Let $A$ be a separable noether normalization of $\Lambda$ and set $I=\ann_{A}\Ext_A^{1}(\Lambda,\syz_{A}\Lambda)$.  Since $A$ is regular, for any associated prime ideal $\fp$ of $A$ the ring $A_{\fp}$ is a field, so that 
\[
\Ext_{A}^{1}(\Lambda,\syz_{A}\Lambda)_{\fp}\cong \Ext_{A_{\fp}}^{1}(\Lambda_{\fp},(\syz_{A}\Lambda)_{\fp}) = 0\,.
\]
Thus $I$ contains a non-zero element, even a non-zerodivisor, of $A$; this element will also be non-zero in $\cent\Lambda$, as $A$ is its subring. Since $\ndiff A{\Lambda}$ contains a non-zerodivisor of $\cent\Lambda$, by Lemma~\ref{le:ndiff-separable}, the desired result follows from Proposition~\ref{pr:noether-normalization}.
\end{proof}

As noted before, the results in this section extend work in \cites{PopescuRoczen90, Wang94, Wang99, Yoshino87}; please see \cite{IyengarTakahashi15}  for details, and further developments in this direction.

\section{Generators for $\mod\Lambda$}
\label{se:generators}

In the preceding section the noether different was used as a tool to find cohomology annihilators. In this section, we take a different tack; one that is inspired by Lemma~\ref{le:ca-module} that gives a method for finding annihilators of cohomology with respect to a single module. The idea then is to find a generator, say $G$, in a sense made precise below, for $\mod\Lambda$ and use the annihilator of $\Ext^{1}_{\Lambda}(G,\syz G)$ to find cohomology annihilators for all of $\mod\Lambda$.

The arguments involve a construction of an ascending chain of subcategories built out of a single module, introduced by Dao and Takahashi~\cite{DaoTakahashi11}. It is an analogue of a construction from Bondal and Van den Bergh~\cite{BondalVandenBergh03} for triangulated categories.

Throughout $\Lambda$ will be a noetherian ring.

\subsection*{Generation in $\mod \Lambda$}
Let $\mcX$ be a subcategory of $\mod\Lambda$. As usual, $\add \mcX$ will denote the subcategory of $\mod \Lambda$ consisting of direct summands of finite direct sums of copies of the modules in $\mcX$.

\begin{definition}
\label{de:building}
We consider an ascending chain of subcategories of $\mod\Lambda$ built out of $\mcX$ as follows: Set $|\mcX|_0:=\{0\}$ and $|\mcX|_{1}=\add \mcX$. For $n\ge2$, let $|\mcX|_n$ be the subcategory of $\mod \Lambda$ consisting of modules $M$ that fit into an exact sequence
\begin{equation}
\label{eq:building}
0 \lra Y \lra M\oplus W \lra X \lra 0
\end{equation}
with $Y$ in $|\mcX|_{n-1}$ and $X$ in $\add\mcX$; said otherwise, $M$ is a direct summand of a $\Lambda$-module $Z$ that admits a finite filtration $\{0\}=Z_{0}\subseteq Z_{1}\subseteq \cdots \subseteq Z_{n}=Z$ with sub-quotients $Z_{i+1}/Z_{i}$ in $\add\mcX$ for each $0\leq i<n$. Clearly, one gets a tower of subcategories  of $\mod \Lambda$:
\[
\{0 \}=|\mcX|_{0} \subseteq  \cdots \subseteq |\mcX|_{n} \subseteq  |\mcX|_{n+1} \subseteq \cdots  \,.
\]
\end{definition}

In what follows, the focus is building modules out of syzygies of a given module. With that in mind, given a subcategory $\mcX\subseteq \mod\Lambda$ and integer $s\ge 0$, we set
\[
\syz^{s}_{\Lambda}\mcX:=\{\syz^{s}_{\Lambda}M\mid M\in \mcX\} \quad\text{and}\quad 
 	\syz^{*}_{\Lambda}\mcX:=\bigcup_{s\ges 0}\syz^{s}_{\Lambda}\mcX\,,
\]
viewed as subcategories of $\mod\Lambda$. Since syzygies are only well-defined up to projective summands, it will be tacitly assumed that $\syz^{*}\mcX$ contains all finitely generated projective $\Lambda$-modules. When $\mcX$ consists of a single module, say $G$, we write $\syz^{s}_{\Lambda}G$ and $\syz^{*}_{\Lambda}G$.

\begin{lemma}
\label{le:mod-ca}
Let $G$ be a $\Lambda$-module and set $I=\ann_{\cent\Lambda}\Ext^{1}_{\Lambda}(G,\syz G)$. For each integer $n\ge 0$ and $\Lambda$-module $M$ in ${\sad G}_{n}$, one then has
\[
I^{n}\cdot \Ext^{\ges 1}_{\Lambda}(M,\Mod \Lambda) = 0\,.
\]
In particular, if $\syz^{s}_{\Lambda}(\mod \Lambda)\subseteq {\sad G}_{n}$ for some integer $s$, then $\cande {s+1}{\Lambda} \supseteq  I^{n}$.
\end{lemma}

\begin{proof}
We induce on $n$. The basis step $n=1$ follows from Remark~\ref{re:ext-and-syzygy} and Lemma~\ref{le:ca-module}, for they yield that $I$ annihilates $\Ext^{\ges 1}_{\Lambda}(\syz^{i}_{\Lambda}G,\Mod\Lambda)$ for any $i$, and so also  $\Ext^{\ges 1}_{\Lambda}(M,\Mod\Lambda)$, for any $M$ in ${\sad G}_{1}$.

Assume $n\ge 2$ and that the desired conclusion holds for all integers less than $n$. Then, for any $i\ge 1$, applying $\Ext^{i}_{\Lambda}(-,N)$ to the sequence \eqref{eq:building} defining $M$ yields an exact sequence 
\[
\Ext^{i}_{\Lambda}(X,N) \lra \Ext^{i}_{\Lambda}(M,N) \oplus \Ext^{i}_{\Lambda}(W,N) \lra \Ext^{i}_{\Lambda}(Y,N)
\]
of $\cent\Lambda$-modules. By the induction hypotheses, $I^{n-1}$ annihilates the module on the right, while $I$ annihilates the one on the left. The exactness of the sequence above implies $I^{n}$ annihilates $\Ext^{i}_{\Lambda}(M,N)$ as desired.

Suppose that $\syz^{s}_{\Lambda}(\mod \Lambda)\subseteq {\sad G}_{n}$ and fix $M$ and $N$ in $\mod \Lambda$. Then $\syz^{s}M$ is in ${\sad G}_{n}$,  so it follows from the already established part of the result that, for each integer $i\ge 1$, one has the second equality below
\[
I^{n}\Ext^{s+i}(M,N) = I^{n}\Ext^{i}(\syz^{s}M,N) =0\,.
\]
The first one is standard; see Remark~\ref{re:ext-and-syzygy}.
\end{proof}

Concerning the hypothesis in the preceding result, we note that there is an equality ${\sad\mcX}_{n}=[\mcX]_n$, where the subcategory $[\mcX]_n$ has been introduced in \cite{DaoTakahashi11}.

\subsection*{Finitistic global dimension}
We introduce the \emph{finitistic global dimension} of a noether algebra $\Lambda$ as the number
\[
\sup\{\gldim\Lambda_{\fp}\mid \fp\in\Reg\Lambda\}\,.
\]
This can be infinite, as is the case for any (commutative noetherian) regular ring of infinite Krull dimension; Nagata~\cite{Nagata62}*{Appendix, Example 1} has constructed such examples.

The result below sums up the discussion in this section up to this point.

\begin{theorem}
\label{th:generator-sing}
Let $\Lambda$ be a noether algebra whose finitistic global dimension is at most $d$. If there exists a $G$ in $\mod \Lambda$ and a non-negative integer $s$ such that $\syz^{s}_{\Lambda}(\mod \Lambda)\subseteq {\sad G}_{n}$ for some $n\ge 0$, then  for $I=\ann_{\cent\Lambda}\Ext^{d+1}_{\Lambda}(G,\syz^{d+1}_{\Lambda} G)$, there are equalities
\[
\Sing\Lambda = \mcV(\can{\Lambda}) = \mcV(\cande {s+d+1}{\Lambda})  = \mcV(I)\,,
\]
In particular, $\Sing\Lambda$ is a closed subset of $\Spec\cent\Lambda$.
\end{theorem}

\begin{proof} Using Remark \ref{re:syzygies}, it is easy to verify that the hypothesis on $G$ yields
\[
\syz^{s+l}(\mod \Lambda)\subseteq |\syz^{*}(\syz^{l}G)|_{n}\quad\text{for any $l\ge 0$.}
\]
Noting that $\Ext^{1}_{\Lambda}(\syz^{d}G,\syz(\syz^{d}G))\cong \Ext^{d+1}_{\Lambda}(G,\syz^{d+1}G)$ as $\cent\Lambda$-modules,  Lemma~\ref{le:mod-ca}, applied to $\syz^{d}G$, thus yields the last inclusion below
\[
\Sing \Lambda \subseteq \mcV(\can{\Lambda})\subseteq \mcV(\cande{s+d+1}{\Lambda})\subseteq  \mcV(I)\,.
\]
The first one is from Lemma~\ref{le:ca-localize}(2) while the second one is by definition of the ideals in question. It thus remains to verify that $\mcV(I)\subseteq \Sing \Lambda$, that is to say that $I\not\subseteq \fp$ for any $\fp$ in $\Reg\Lambda$. Recall that for any finitely generated module $E$ over a commutative ring $R$ and $\fp$ in $\Spec R$, one has $E_{\fp}=0$ if and only if $\ann_{R}E \not\subseteq\fp$. Thus, the desired conclusion follows because for any $\fp$ in $\Reg\Lambda$ there are isomorphisms
\[
\Ext^{d+1}_{\Lambda}(G,\syz^{d+1} G)_{\fp}  \cong \Ext^{d+1}_{\Lambda_{\fp}}(G_{\fp},(\syz^{d+1} G)_{\fp}) \cong 0\,,
\]
where the last one holds because $\gldim\Lambda_{\fp}\le d$, by hypotheses.
\end{proof}

\begin{remark}
The proof of Theorem~\ref{th:generator-sing} gives a more precise result: If $G$ can be chosen to be an $i$th syzygy module, for some $i\ge 0$, then $\mcV(\can{\Lambda}) = \mcV(\cande {s+d-i+1}{\Lambda})$.
\end{remark}

\subsection*{Strong generators for $\mod\Lambda$}
Let $\Lambda$ be a noetherian ring. We say that a finitely generated $\Lambda$-module $G$ is a \emph{strong generator} for $\mod\Lambda$ if the following conditions holds:
\begin{enumerate}
\item $\Lambda$ is a direct summand of $G$, and
\item there exist non-negative integers $s$ and $n$ such that $\syz^{s}_{\Lambda}(\mod \Lambda)\subseteq |G|_{n}$.
\end{enumerate}
The first condition says that $G$ is a generator for $\mod\Lambda$ in the usual sense of the word. This definition is motivated by later considerations, especially the results on Section~\ref{se:commutative}. Observe that the criterion for $G$ to be a strong generator is stronger than the conclusion of Theorem~\ref{th:generator-sing}, for it does not allow for the syzygies of $G$.

We now reconcile this notion with one based on thick subcategories of $\mod\Lambda$.

Let $\mcX$ be a subcategory of $\mod\Lambda$. Set $\thick^{0}(\mcX):=\{0\}$ and $\thick^1(\mcX):=\add\mcX$.  For $n\ge2$ let $\thick^n(\mcX)$ be the subcategory of $\mod \Lambda$ consisting of direct summands of any module that appears in an exact sequence
\[
0 \to X \to Y \to Z \to 0
\]
such that, among the other two modules, one is in $\thick_{\Lambda}^{n-1}(\mcX)$ and the other is in $\add\mcX$.

\begin{proposition}
\label{pr:strong-generation}
Let $\mcX$ and $\mcY$ be subcategories of $\mod\Lambda$. The following statements hold for each integer $n\ge1 $.
\begin{enumerate}[{\quad\rm(1)}]
\item
$|\mcX|_{n}\subseteq \thick^{n}(\mcX)$.
\item
$\syz^{n-1}_{\Lambda}(\thick^{n}(\mcX))\subseteq |\bigcup_{i=0}^{2(n-1)}\syz^{i}_{\Lambda}\mcX|_{n}$.
\item
If $\syz^{s}_{\Lambda}(\mcY)\subseteq \thick^{n}(\mcX)$ for some integer $s\ge 1$, then $\mcY\subseteq \thick^{n+s}(\mcX\cup\{\Lambda\})$.
\end{enumerate}
\end{proposition}

\begin{proof}
The inclusion in (1) is immediate from definitions, as is (3) for when $\syz^{s}(M)$ is in $\thick^{n}(\mcX)$, it follows from the exact sequence 
defining the syzygy module:
\[
0\lra \syz^{s}M\lra P_{s-1}\lra \cdots \lra P_{0}\lra M\lra 0\,,
\]
with each $P_{i}$ a finitely generated projective, that $M$ is in $\thick^{n+s}(\mcX\cup\{\Lambda\})$.

(2) For each  $n\ge 0$, set $\mcC_{n}=\bigcup_{i=0}^{n}\syz^{i}\mcX$. The desired statement is that  
\[
\syz^{n-1}(\thick^{n}(\mcX)) \subseteq |\mcC_{2(n-1)}|_{n} \quad\text{for each $n\ge 1$.}
\]
We verify this by an induction on $n$.  The base case $n=1$  is clear, for both $\thick^{1}(\mcX)$ and $|\mcC_{0}|_{1}$ are $\add(\mcX)$. Assume that the  inclusion holds for some $n\ge 1$, and for every subcategory $\mcX$ of $\mod \Lambda$.

Since both $\thick^{n+1}(\mcX)$ and $\mcC_{n+1}$ are closed under direct summands, it suffices to verify that given $X$ and $Y$ in $\mod\Lambda$ with one in $\thick^{1}(\mcX)$ and the other in $\thick^{n}(\mcX)$, and an exact sequence of $\Lambda$-modules of one of the following types:
\begin{gather*}
0\lra W\lra X \lra Y\lra 0  \tag{i} \\
0\lra X\lra W \lra Y\lra 0  \tag{ii} \\
0\lra X \lra Y \lra W\lra 0  \tag{iii}
\end{gather*}
the $\Lambda$-module $\syz^{n}W$ is in $|\mcC_{2n}|_{n+1}$. This can be verified by a direct case-by-case analysis; there are six cases to consider, depending on whether $X$ is in $\thick^{1}(\mcX)$ or in $\thick^{n}(\mcX)$. 

We do this when $X$ is in $\thick^{1}(\mcX)$ and $Y$ is in $\thick^{n}(\mcX)$; the argument in the other case is  analogous. By the induction hypothesis $X$ is in $|\mcC_{0}|_{1}$ and $\syz^{n-1}Y$ is in $|\mcC_{2(n-1)}|_{n}$, and then it is easy to verify that
\[
\syz^{n-1}X, \syz^{n}X\in |\mcC_{2n}|_{1} \quad\text{and}\quad
\syz^{n}Y, \syz^{n+1}Y\in |\mcC_{2n}|_{n}\,.
\]
These remarks  will be used without further comments in what follows.

\medskip

Case (i): The exact sequence (i) gives rise to an exact sequence
\[
0\lra \syz Y\lra W\oplus P \lra X\lra 0
\]
for some finitely generative projective module $P$. Since $\syz^{n}X$ is in $|\mcC_{2n}|_{1}$ and $\syz^{n+1} Y$ is in $ |\mcC_{2n}|_{n}$ the exact sequence above yields that $\syz^{n}W$ is in $|\mcC_{2n}|_{n+1}$, as desired.

\medskip

Case (ii): Since $\syz^{n}X$ is in  $|\mcC_{2n}|_{1}$ and $\syz^{n}Y$ is in  $|\mcC_{2n}|_{n}$, it follows from the exact sequence  (ii) that $\syz^{n}W$ is in $|\mcC_{2n}|_{n+1}$.

\medskip

Case (iii): The exact sequence (iii) gives rise to an exact sequence
\[
0\lra \syz Y \lra \syz W  \lra X \oplus P \lra 0
\]
where $P$ is a finitely generated projective module. Since $\syz^{n-1}X$ is in  $|\mcC_{2n}|_{1}$ and $\syz^{n}Y$ is in  $|\mcC_{2n}|_{n}$, the desired result holds.
\end{proof}

The following result is an immediate consequence of Proposition~\ref{pr:strong-generation}.

\begin{corollary}
\label{co:strong-generation}
Let $G$ be a finitely generated $\Lambda$-module.
\begin{enumerate}[{\quad\rm(1)}]
\item 
If  $\syz^{s}(\mod\Lambda)\subseteq |G|_{n}$ for some positive integers $s,n$, then $\mod\Lambda = \thick^{s+n}(G\oplus\Lambda)$.
\item
If  $\mod\Lambda=\thick^{n}(G)$ for an integer $n\ge1$, then $\syz^{n-1}(\mod\Lambda)\subseteq |\bigcup_{i=0}^{2(n-1)}\syz^{i}G|$. \qed
\end{enumerate}
\end{corollary}

\begin{remark}
\label{re:generation}
The import of the preceding result is that $\mod\Lambda$ has a strong generator if and only if there exists a finitely generated module such that the thick subcategory it generates is all of $\mod\Lambda$. The latter condition is akin to the one for a strong generator of the  bounded derived category, $\dbcat(\mod \Lambda)$, as a triangulated category. However, for our applications the notion of a strong generator adopted here is the better one, for it distinguishes between a module and its syzygy. This also suggest that for applications to module theory it would be useful to investigate the set of pairs $(s,n)$ of integers for which there exists a $G$ in $\mod\Lambda$ with $\syz^{s}(\mod\Lambda)\subseteq |G|_{n}$. This is a two parameter version of the Orlov spectrum of a triangulated category; see \cite{Orlov09}*{Definition~3}.
\end{remark}

\subsection*{A compactness argument}
Let $\mcX$ be a subcategory of $\mod\Lambda$. Let $\Add{\mcX}$ denote the subcategory of $\Mod \Lambda$ consisting of direct summands of arbitrary direct sums of copies of the modules in $\mcX$. Following the construction in Definition~\ref{de:building}, one gets a tower 
\[
\{0\} = |\Add{\mcX}|_{0} \subseteq\cdots  \subseteq |\Add{\mcX}|_{n} \subseteq |\Add{\mcX}|_{n+1} \subseteq \cdots
\]
of subcategories of $\Mod \Lambda$.  The result below is  a module theoretic version of~\cite{BondalVandenBergh03}*{Proposition 2.2.4}. It will be used in the sequel to prove that rings essentially of finite type over a field have strong generators.

\begin{lemma}
\label{le:module-building}
Let $\Lambda$ be a noetherian ring and $\mcX\subseteq\mod \Lambda$ a subcategory. For each integer $n\ge 1$ there is an equality
$|\Add{\mcX}|_n\,\cap\,\mod \Lambda=|\mcX|_n$.
\end{lemma}

\begin{proof}
First we verify the following claim; confer \cite{Rouquier08a}*{Proposition~3.13}

\begin{claim}
Let $M$ be a finitely generated $\Lambda$-module and $\vf\colon M\to Z$ a homomorphism in $\Mod\Lambda$, where $Z$ admits a filtration $\{0\}=Z_{0}\subseteq Z_{1}\subseteq\cdots \subseteq Z_{n}=Z$ with sub-quotients in $\Add\mcX$. Then $\vf$ factors as $M\to W\to Z$ where $W$ is a finitely generated $\Lambda$-module with a filtration
\[
\xymatrixcolsep{.75pc}
\xymatrixrowsep{1.5pc}
\xymatrix{ 
Z_{0} \ar@{^{(}->}[rr] && Z_{1} \ar@{^{(}->}[rr] &&\cdots \cdots  \ar@{^{(}->}[rr]&& Z_{n}\ar@{=}[r] & Z \\
W_{0} \ar@{->}[u] \ar@{^{(}->}[rr] && W_{1} \ar@{->}[u] \ar@{^{(}->}[rr] &&\cdots \cdots  \ar@{^{(}->}[rr] && W_{n} \ar@{->}[u] \ar@{=}[r] & W 
}
\]
where $W_{i+1}/W_{i}$ is a direct summand of $Z_{i}/Z_{i-1}$ and in $\add\mcX$, for each $0\leq i < n$.
\end{claim}

Indeed, set $Z^{i}=Z_{i}/Z_{i-1}$ for each $i\ge 1$; this is in $\Add\mcX$, by hypotheses. Set $M_{n}=M$ and $\vf_{n}=\vf$.
We construct, for $1\le i\le n$, commutative diagrams with exact rows
\[
\xymatrixcolsep{2.5pc}
\xymatrix{
0 \ar@{->}[r] & Z_{i-1} \ar@{<-}[d]_{\vf_{i-1}} \ar@{->}[r]
                    &Z_i \ar@{<-}[d]_{(\vf_{i},\kappa_{i})}\ar@{->}[r] &Z^{i}\ar@{<-}[d] \ar@{->}[r] & 0 \\
0 \ar@{->}[r] &M_{i-1} \ar@{->}[r]  & M_{i} \oplus F^{i} \ar@{->}[r]^-{(\psi_{i},\eps_{i})} &W^{i} \ar@{->}[r] & 0
}
\]
where $W^{i}$ is a direct summand of $Z^{i}$ and in $\add\mcX$, and $F^{i}$ is a free $\Lambda$-module of finite rank. These are obtained as follows: Since $M_{n}$ is finitely generated, the composed map $M_{n}\to Z_{n}\to Z^{n}$ factors as $M_{n}\xra{\psi_{n}} W^{n}\to Z^{n}$, where $W^{n}$ is a direct summand of $Z^{n}$ and in $\add\mcX$. Choose a surjective map $\eps_{n} \colon F^{n}\to W^{n}$, with $F^{n}$ a finitely generated free $\Lambda$-module, and a lifting $\kappa_{n}\colon F^{n}\to Z_{n}$ of the composition $F^{n}\to W^{n}\to Z^{n}$, through the surjection $Z_{n}\to Z^{n}$. Setting $M_{n-1}$ to be the kernel of $(\psi_{n},\eps_{n})$ and $\vf_{n-1}$ the induced map gives the data required to construct the diagram above for $i=n$; it is readily seen to be commutative. Observe that $M_{n-1}$ is finitely generated; now repeat the construction above for $\vf_{n-1}$.

Finally, set $W_{i} := M_{i}\oplus F^{i}\oplus F^{i-1}\oplus \cdots \oplus F^{1}$ for each $1\leq i\leq n$ and $W_{0}:=M_{0}$. There is then a canonical inclusion $W_{i-1}\subseteq W_{i}$, with quotient $W^{i}$, and the desired commutative diagram is as follows:
\[
\xymatrix{
Z_{0} \ar@{<-}[d]_{\vf_{0}}\ar@{^{(}->}[r] &  Z_{1} \ar@{<-}[d]_{(\vf_{1},\kappa_{1})}\ar@{^{(}->}[r] 
	&\cdots\cdots \ar@{^{(}->}[r] & Z_{n} \ar@{<-}[d]_{(\vf_n,\kappa_n,\dots,\kappa_{1})} \\
W_{0} \ar@{^{(}->}[r] & W_{1}  \ar@{^{(}->}[r] & \cdots \cdots \ar@{^{(}->}[r] & W_{n}
}
\]
It remains to note that the composition $M_{n}\to W_{n}\to Z_{n}$ is precisely $\vf$. 

This justifies the claim.  Now suppose that $M$ is a finitely generated module in $|\Add\mcX|_{n}$, so that there exists a split monomorphism $\vf\colon M\to Z$, where $Z$ has a $n$-step filtration with sub-quotients in $\Add(\mcX)$. Thus the claim applies, and we set $C$ to be the cokernel of  the inclusion $W_0\subseteq W$; it is in $|\mcX|_n$, by construction. As $Z_0=\{0\}$, the map $W\to Z$ factors through $C$. Since $M$ is a direct summand of $W$, it is also a direct summand of $C$. Therefore $M$ is in $|\mcX|_{n}$, as desired. 
\end{proof}

\section{Commutative rings}
\label{se:commutative}
The focus of this section is on converses to Theorem~\ref{th:generator-sing} that justify the claim made in the introduction; namely,  the existence of cohomology annihilators is intertwined with the existence of strong generators for module categories. Though the fundamental result in this section, Theorem~\ref{th:ca-generator}, can be formulated for noether algebras, we have chosen to present it for commutative rings, for the statement appears most natural in that context. The proof is an adaptation of \cite{DaoTakahashi11}*{Theorem~5.7}.

\begin{theorem}
\label{th:ca-generator}
Let $R$ be a commutative noetherian ring of Krull dimension $d$. If there exists a positive integer $s$ such that $\cande s{R/\fp}\ne 0$ for each $\fp \in \Spec R$, then there is a finitely generated $R$-module $G$ and an integer $n$ such that $\syz^{s+d-1}_{R}(\mod R)\subseteq |G|_{n}$. In particular, $\mod R$ has a strong generator.
\end{theorem}

\begin{proof}
We induce on $\dim R$, the base case $\dim R=0$ being clear for then $\mod R = |R/J(R)|_{l}$ where $J$ is the Jacobson radical of $R$ and $l$ is its Loewy length. Assume $\dim R\ge 1$. 

Consider first the case when $R$ is a domain. Then, by hypothesis, $\cande sR$ contains a non-zero element $a$; we can assume that it  is not invertible, for $\dim R\ge 1$. By the induction hypothesis, there exists an $R/aR$-module $G$ such that 
\begin{equation}
\label{eq:reduction}
\syz^{s+d-2}_{R/aR}(\mod R/aR)\subseteq |G|_{n}\quad\text{for some $n\in\bbZ$.}
\end{equation}
Fix an $R$-module $M$ and set $N=\syz^{s+d-1}_{R}M$. Since $R$ is a domain, there is an isomorphism
\[
N/aN \cong \syz^{s+d-2}_{R/aR}(\syz_{R}M/a\syz_{R}M)\,;
\]
see, for example, \cite{DaoTakahashi11}*{Lemma~5.6}. Viewing $G$ as an $R$-module, it follows from \eqref{eq:reduction} that $N/aN$ is in $|G|_{n}$, and hence that $\syz_{R}(N/aN)$ is in $|\syz_{R}G|_{n}$.  Since $\Ext^{1}_{R}(N,-)$ is isomorphic to $\Ext^{s}(\syz^d_{R}M,-)$, the element
$a$ annihilates it. Since $N$ is at least a first syzygy, because $s+d-1\ge 1$, and $R$ is a domain, $a$  is a non-zerodivisor on $N$. Therefore Remark~\ref{re:ext2} yields that $N$ is a direct summand of $\syz_{R}(N/aN)$. In conclusion $N$, that is to say, $\syz^{s+d-1}_{R}(M)$ is in $|\syz_{R}(G)|_{n}$. Since $M$ was arbitrary, this gives the desired result for $R$.

This completes the proof when $R$ is a domain. When it is not, one can choose ideals $(0)=I_{0}\subset I_{1}\subset\cdots \subset I_{m}=R$ such that $I_{i+1}/I_{i}\cong R/\fp_{i}$ for some $\fp_{i}$ in $\Spec R$, for each $i$.  By the already established part of the result, there exists an integer $n$ and $R$-modules $G_{i}$ such that $\syz_{R/\fp_{i}}^{s+d-1}(\mod R/\fp_{i})\subseteq |G_{i}|_{n}$. For any $M$ in $\mod R$, there are exact sequences
\[
0\lra I_{i}M \lra I_{i+1}M \lra I_{i+1}M/I_{i}M \lra 0 
\]
for $0\le i\le m-1$. It then follows from \cite{DaoTakahashi11}*{Corollary~5.5}  that there exists a $G$ in $\mod R$ such that $\syz^{s+d-1}_{R}M$ is in $|G|_{m(s+d)(n+1)}$ for each $M\in\mod R$. 
\end{proof}

Under stronger hypotheses, the argument in the proof of the theorem above gives the following, more precise, statement.

\begin{theorem}
\label{th:ca-generator2}
Let $R$ be a commutative noetherian ring of Krull dimension $d$. If for each prime ideal $\fp$ in $R$, there exists an integer $s\le \dim R/\fp+1$ such that $\cande s{R/\fp}\ne 0$ , then there exists a $G$ in $\mod R$ and an integer $n$ such that $\syz^{d}_{R}(\mod R)\subseteq |G|_{n}$.\qed
\end{theorem}

\subsection*{Excellent local rings}
For complete local rings, the part of the result below concerning the existence of the generator $G$ was proved in \cite{DaoTakahashi11}*{Theorem 5.7} under the additional hypothesis that the residue field of $R$ is perfect. The description of the singular locus should be compared with \cite{Wang94}*{Corollary 5.15}. The latter result implies that when $R$ is an equidimensional complete local ring containing a perfect field, 
one can replace $2d+1$ by $d+1$ in the statement below. We refer the reader to \cite{Matsumura80}*{(34.A)} for the notion of excellence.
 
\begin{theorem}
\label{th:complete-equicharacteristic}
Let $R$ be an equicharacteristic excellent local ring of Krull dimension $d$. There are  equalities
\[
\mcV(\can R)=\mcV(\cande {2d+1}R) = \Sing R\,.
\]
Furthermore, there exists a finitely generated $R$-module $G$ and an integer $n$ such that $\syz^s(\mod R)\subseteq |G|_n$ for $s=3d$; if $R$ is complete there is a $G$ for which $s=d$ suffices.  In particular, $\mod R$ has a strong generator.
\end{theorem}

\begin{proof}
First we verify the result when $R$ is complete. Fix a prime ideal $\fp$ in $R$. The ring $R/\fp$ is then a complete equicharacteristic local domain, and hence has a separable noether normalization: When the residue field of the ring is perfect (for example, when its characteristic is $0$), this is due to Cohen~\cite{Cohen46}*{Theorem~16}); the general positive characteristic case is a result of Gabber's; see~\cite{Gabber13}*{IV, Th\'eor\`eme 2.1.1}. Thus Theorem~\ref{th:separable-noether-normalization} applies and yields that $\cande {\dim R/\fp+1}{R/\fp}$ is non-zero. Since $\fp$ was arbitrary, Theorem~\ref{th:ca-generator2} guarantees the existence of a module $G$ with stated property. Given this Theorem~\ref{th:generator-sing} justifies the equalities concerning the singular locus of $R$.

Next we verify this equality for a general excellent local ring. Given Lemma~\ref{le:ca-localize}(2), the moot point is that $\mcV(\cande{2d+1}R)\subseteq \Sing R$ holds. As $R$ is excellent, $\Sing R$ is a closed subset of $\Spec R$, by definition; see \cite{Matsumura80}*{Definition (34.A) and (32.B)}. Let $I$ be an ideal in $R$ with $\mcV(I)=\Sing R$. Let $\vf\colon R\to \widehat R$ denote completion with respect to the maximal ideal of $R$, and ${}^{\sfa}\vf\colon \Spec{\widehat R}\to \Spec R$ the induced map. There are then equalities
\[
\mcV(I\widehat R) = {}^{\sfa}\vf^{-1}(\mcV(I))=\Sing\widehat R =\mcV(\cande{2d+1}{\widehat R})
\]
where the first one is standard, the second one holds because the fibers of $\vf$ are regular, by \cite{Matsumura80}*{Theorem 79}, and the last is from the already established part of the result, applied to the complete local ring $\widehat R$. It follows that, for some non-negative integer $n$, there is an inclusion $I^{n}\widehat R\subseteq \cande{2d+1}{\widehat R}$.  Fix finitely generated $R$-modules $M,N$. Since $\widehat R$ is flat as an $R$-module, the natural map 
\[
\Ext^{2d+1}_{R}(M,N)\otimes_{R}\widehat R\to \Ext^{2d+1}_{\widehat R}(M\otimes_{R}\widehat R,N\otimes_{R}\widehat R)
\]
is an isomorphism,  so we deduce that $I^{n}$ annihilates the module on the left, and hence also $\Ext^{2d+1}_{R}(M,N)$, because $\widehat R$ is also faithful as an $R$-module. In summary, $I^{n}\subseteq \cande{2d+1}R$. This gives the desired inclusion.

It remains to justify the existence of a $G$ with the stated properties. As in the first part of the proof, given Theorems~\ref{th:ca-generator}  it suffices to note that for $\fp$ in $\Spec R$, the ideal $\cande{2d +1}{R/\fp}$ is non-zero: we already know that the closed subset of $\Spec (R/\fp)$ that it defines  coincides with $\Sing(R/\fp)$, and that is a proper closed subset because $R/\fp$ is a domain.
\end{proof}

\subsection*{Rings essentially of finite type over fields}
Compare the next result with Theorem~\ref{th:complete-equicharacteristic}. When $k$ is perfect and $R$ is itself a finitely generated $k$-algebra and a domain, the equalities below can be improved:  $2d+1$ can be replaced by $d+1$; this is by \cite{Wang99}*{Theorem~3.7}.

\begin{theorem}
\label{th:eft}
Let $k$ be a field and $R$ a localization of a finitely generated $k$-algebra of Krull dimension $d$.  There are equalities
\[
\mcV(\can R)=\mcV(\cande {2d+1}R) = \Sing R\,.
\]
Furthermore, there is a finitely generated $R$-module $G$ such that $\syz^d(\mod R)\subseteq |G|_n$ for some integer $n$. In particular, $\mod R$ has a strong generator.
\end{theorem}

\begin{proof}
It suffices to prove that a $G$ as desired exists, for then one can invoke Theorem~\ref{th:generator-sing} to  justify the stated equalities.
Suppose $R=U^{-1}A$ for some finitely generated $k$-algebra $A$ of Krull dimension $d$, and multiplicatively closed subset $U$ of $A$.  
Since every finitely generated $R$-module is a localization of a finitely generated $A$-module, and localization preserves exact sequences, it suffices to prove the result for $A$. Thus,  we may assume $R$ is itself a finitely generated $k$-algebra, of dimension $d$.

First, we consider the case where $k$ is perfect. The argument that such a $G$ exists is then the same as for Theorem~\ref{th:complete-equicharacteristic}: For each prime ideal $\fp$ in $R$, the ring $R/\fp$ is also a finitely generated $k$-algebra, and hence has a separable noether normalization (see, for example,~\cite{Nagata62}*{Theorem 39.11}) so Theorem~\ref{th:separable-noether-normalization} yields that $\cande {\dim R/\fp+1}{R/\fp}$ is non-zero. The desired result then follows from Theorem~\ref{th:ca-generator2}.

Next we tackle the case of an arbitrary field $k$ by adapting an argument from \cite{KellerVandenBergh08}*{Proposition 5.1.2}.
Let $K$ be an algebraic closure of $k$. Then $R\otimes_kK$ is a finitely generated $K$-algebra of dimension $d$. Since $K$ is perfect, the already establish part of the statement yields a finitely generated $R\otimes_kK$-module $C$ and an integer $n>0$ such that 
\[
\syz^d(\mod (R\otimes_kK))\subseteq|C|_n\,.
\]
Since $C$ is finitely generated, there exists a finite field extension $l$ of $k$ and a finitely generated $(R\otimes_kl)$-module $G$ such that $C\cong G\otimes_lK$. Note that $R\otimes_{k}l$, hence also $G$, is finitely generated as an $R$-module. We claim that, viewing $G$ as an $R$-module, there is an inclusion
\[
\syz^d(\mod R)\subseteq|G|_n\,.
\]
Indeed, let $M$ be a finitely generated $R$-module. Then the  $(R\otimes_{k}K)$-module $M\otimes_{k}K$ is finitely generated and hence $\syz_{R\otimes_kK}^d(M\otimes_kK)$ belongs to $|G\otimes_lK|_n$. Since the $R$-module $G\otimes_lK$ is a (possibly infinite) direct sum of copies of $G$,  there is an inclusion 
\[
|G\otimes_lK|_n\subseteq |\Add\{G\}|_n \quad\text{in $\Mod R$}.
\]
The $R$-module $\syz_R^dM$ is a direct summand of $(\syz_R^dM)\otimes_kK\cong \syz_{R\otimes_kK}^d(M\otimes_kK)$, so it follows that $\syz_R^dM$ is in $|\Add\{G\}|_n\cap\mod R$. It remains to apply Lemma~\ref{le:module-building}, recalling that $G$ is finitely generated over $R$ as well. This completes the proof of the theorem.
\end{proof}

The rings in Example~\ref{ex:non-open-loci} and the results in this section suggest the following.

\begin{question}
\label{qu:excellent}
Does $\mod R$ have a strong generator when $R$ is an excellent ring?
\end{question}

Any counter-example must have Krull dimension at least two; this follows from Corollary~\ref{co:nagata1} below, which is deduced from the next statement. In it, the Jacobson radical and the nil radical of a ring $R$ are denoted $\rad R$ and $\nil R$, respectively. See \cite{Matsumura80}*{\S31} for the definition of a Nagata ring.

\begin{proposition}
\label{pr:nagata1}
Let $R$ be a Nagata ring of Krull dimension one.
\begin{enumerate}[\quad\rm(1)]
\item
Assume $R$ is reduced, let $S$ be the integral closure of $R$, and $T$ the quotient of $R$ by its conductor ideal.  With $\ell$ the Loewy length of $T$ and $n= 2 + 2\ell$, there is an inclusion
\[
\syz_{R}(\mod R)\subseteq |S \oplus \syz_R(S) \oplus (T/\rad T) \oplus \syz_R(T/\rad T)|_{n}\,.
\]
\item
Let $\overline R$ denote the ring $R/\nil R$ and $G$ a finitely generated $\overline R$-module with $\syz_{\overline R}(\mod \overline R)$ contained in $|G|_n$ for some integer $n$. For $j= \min\{ i \mid (\nil R)^i=0 \}$  one has
\[
\syz_{R}(\mod R)\subseteq |G \oplus \nil(R)|_{2nj}\,.
\]
\end{enumerate}
\end{proposition}

\begin{proof}
For (1), let $C$ be the conductor of $R$. Since $R$ is a reduced Nagata ring, $S$ is a finitely generated $R$-module, and hence $C$ contains a non-zerodivisor of $R$. Any finitely generated $S$-module is also finitely generated as an $R$-module, and that the ring $T$ is artinian; see, for example. These remarks will be used without comment.

Consider the exact sequence of $R$-modules
\[
0 \lra C \lra R \lra T \lra 0\,.
\]
Fix an $R$-module $M$. Applying $M\otimes_{R}-$ to this exact sequence yields exact sequences
\begin{equation}
\label{eq:nagata1}
\begin{gathered}
0\lra N \lra M \lra M\otimes_{R} T \lra 0\,, \\
0\lra \Tor^{R}_{1}(M,T)\lra M\otimes_{R} C \lra N  \lra 0\,.
\end{gathered}
\end{equation}
Note that $C$ is also an ideal of $S$, so $M\otimes_{R}C$ acquires a structure of an $S$-module.  Since $S$ has global dimension one, one can construct an exact sequence  $0\to P\to Q\to M\otimes_{R}C\to 0$ of $S$-modules with $P$ and $Q$ finitely generated projective $S$-modules. This exact sequence yields an exact sequence of the form
\[
0\lra \syz_{R}(Q) \lra \syz_{R}(M\otimes_{R}C) \lra P\lra 0\,.
\]
It follows that $\syz_{R}(M\otimes_{R}C)$ is contained in the subcategory $|S \oplus \Omega_R(S) |_2$ of $\mod R$. On the other hand, $M\otimes_{R}T$ and $\Tor^{R}_{1}(M,T)$ are in $|T/\rad T|_{\ell }$, because the ring $T$ is artinian.  Given these, the desired result follows from the exact sequences in \eqref{eq:nagata1}.

(2)  For any $R$-module $M$ has a filtration $\{0\}\subseteq I^{j}M \subseteq ...\subseteq IM \subseteq M$ that induces, up to projective summands, a filtration
\[
\{0\}\subseteq \syz_{R}(I^{j}M) \subseteq ...\subseteq \syz_{R}(IM) \subseteq \syz_{R}(M)\,.
\]
The desired assertion follows from \cite{DaoTakahashi11}*{Proposition 5.3(2)}.
\end{proof}

The next result is a direct consequence of the Proposition~\ref{pr:nagata1}.  Since excellent rings have the Nagata property, see  \cite{Matsumura80}*{(34.A)}, it settles, in the affirmative, Question~\ref{qu:excellent} for rings of Krull dimension one.

\begin{corollary}
\label{co:nagata1}
Let R be a Nagata ring of Krull dimension one. Then there exist an $R$-module $G$ and an integer $n$ such that $\syz_{R}(\mod R)\subseteq |G|_n$.\qed
\end{corollary}

\section{Ascent and descent}
\label{se:ascent-descent}

The crucial input in Theorems~\ref{th:ca-generator} and \ref{th:ca-generator2} is the existence of non-zero cohomology annihilators. Motivated by this, in this section we track the ascent and descent of this property between a commutative noetherian ring $A$ and a noether $A$-algebra $\Lambda$, which need not be commutative. The central result is Corollary~\ref{co:ascent-descent-nzd} that is deduced from the more technical, but also more precise, Theorems~\ref{th:descent} and \ref{th:ascent}.

\begin{theorem}
\label{th:descent}
Let $\Lambda$ be a noether $A$-algebra and set $I:=\ann_{A}\Ext^{1}_{A}(\Lambda,\syz_{A}\Lambda)$. For each integer $n\ge 1$ and element $a\in\cande n\Lambda$, one has
\begin{equation}
\label{eq:descent1}
a^{2}I^{n}\, \Ext_{A}^{\ges n}(\mod\Lambda,\mod A)=0\,.
\end{equation}
When the $A$-module $\Lambda$ has positive rank, there exists a non-zerodivisor $b\in A$ such that
\begin{equation}
\label{eq:descent2}
b\,\ann_{A} \Ext_{A}^{\ges n}(\mod\Lambda,\mod A)\subseteq \cande nA\,.
\end{equation}
\end{theorem}

\begin{proof}
We repeatedly use the fact that $I$ annihilates $\Ext^{\ges 1}_{A}(\Lambda, \mod A)$; see Lemma~\ref{le:ca-module}.

Fix an $X\in\mod\Lambda$ and for $Y:=\syz_\Lambda^{n-1}X$ consider an exact sequence of $\Lambda$-modules
\[
0\to \syz_\Lambda(Y/aY) \to P \to Y/aY \to 0
\]
with $P$ projective. For each $M\in\mod A$ and integer $i\ge 1$, it induces an exact sequence
\[
\Ext_A^{i}(P,M) \to \Ext_A^{i}(\syz_\Lambda(Y/aY),M) \to \Ext_A^{i+1}(Y/aY,M)
\]
of $A$-modules, so $a I$ annihilates the module in the middle. Since $a\cdot \Ext^{i}_{\Lambda}(Y,\mod\Lambda)=0$, the exact sequence in Remark~\ref{re:ext2} yields an exact sequence
\[
\Ext_A^{i}(\syz_\Lambda(Y/aY),M) \to \Ext_A^{i}(Y,M)\oplus\Ext_A^{i}(\syz_\Lambda Y,M)\to\Ext_A^{i}((0:_Ya),M)
\]
Therefore, $a^{2}I$ annihilates $\Ext_{A}^{i}(Y,M)$, for each $i\ge 1$. Considering the exact sequence 
\[
0 \to Y \to P_{n-2} \to \cdots \to P_1 \to P_0 \to X \to 0
\]
of $\Lambda$-modules with each $P_i$ projective, a standard iteration yields \eqref{eq:descent1}.

Assume now that $\Lambda$ has positive rank, say equal to $r$, as an $A$-module. Thus, there is an exact sequence of $A$-modules
\[
0 \to A^{r} \to \Lambda \to T \to 0
\]
such that $bT=0$ for some non-zerodivisor $b\in A$. Let $M$ be a finitely generated $A$-modules. The exact sequence above yields an exact sequence of $A$-modules
$$
\Tor_1^A(M,T) \to M^{r} \to M\otimes_A\Lambda \to M\otimes_AT \to 0.
$$
Since $b$ annihilates $\Tor_1^A(M,T)$ and $M\otimes_AT$,  Remark~\ref{re:4terms} below, applied with $N=M\otimes_{A}\Lambda$ and $J=(b)$, gives \eqref{eq:descent2}. This is where the hypothesis that $r\ge 1$ is used.
\end{proof}

\begin{remark}
\label{re:4terms}
Let $K \to M \to N \to C$ be an exact sequence of $\Lambda$-modules and $J$ an ideal in $\cent\Lambda$ with $J(K\oplus C)=0$. For any $\Lambda$-module $X$ and integer $n$ there is an inclusion
\[
J^{2}\cdot \ann_{\cent\Lambda}\Ext^{n}_{\Lambda}(N,X)\subseteq \ann_{\cent\Lambda}\Ext^{n}_{\Lambda}(M,X)\,.
\]
Indeed, replacing $K$ and $C$ by the appropriate quotient module and submodule, we may assume that there are exact sequences $0 \to K \to M \to Z \to 0$ and $0\to Z\to N\to C\to 0$. These induce exact sequences
\begin{gather*}
 \Ext_R^n(N,X) \to \Ext_R^n(Z,X) \to \Ext_R^{n+1}(C,X),\\
 \Ext_R^n(Z,X) \to \Ext_R^n(M,X) \to \Ext_R^n(K,X).
\end{gather*}
Since $J$ annihilates $K$ and $C$, it also annihilates $\Ext_R^{n+1}(C,X)$ and $\Ext_R^n(K,X)$. The desired inclusion follows.
\end{remark}

\begin{lemma}
\label{le:2ndsyzygy}
Let $M$ be a finitely generated $A$-module and set $\Lambda=\End_A(M)$.  Up to projective $\Lambda$-summands, any second syzygy of a finitely generated $\Lambda$-module is isomorphic to a $\Lambda$-module of the form $\Hom_{A}(M,N)$, for some finitely generated $A$-module $N$.
\end{lemma}

\begin{proof}
Consider the functor $\Hom_{A}(M,-)\colon \mod A\to \mod \Lambda$. It restricts to an equivalence $\add_RM\xra{\ \sim\ } \proj\Lambda$, so that each finitely generated projective $\Lambda$-module is isomorphic to $\Hom_R(M,M')$ for some $M'\in\add_RM$. Moreover, for any  finitely generated $\Lambda$-module $X$, one can construct an exact sequence  of $\Lambda$-modules
\[
\Hom_A(M,M_1) \xra{\ \Hom_A(M,f)\ } \Hom_A(M,M_0) \lra X \lra 0
\]
where $f\colon M_{1}\to M_{0}$ is an $A$-linear map in  $\add_AM$. Observe that the inclusion $\Ker(f) \subseteq M_1$ induces  an exact sequence of $\Lambda$-modules
\[
0 \lra \Hom_A(M,\Ker(f)) \lra \Hom_A(M,M_1) \xra{\ \Hom_{A}(M,f)\ } \Hom_{A}(M,M_0) \lra X \lra 0\,.
\]
This shows that $\Hom_A(M,\Ker(f))$ is a second syzygy of $X$ as a $\Lambda$-module.
\end{proof}

\begin{theorem}
\label{th:ascent}
Let $M$ be a finitely generated $A$-module and set $I:=\ann_{A}\Ext^{1}_{A}(M,\syz_{A}M)$. When $M$ has positive rank,
there exists a non-zerodivisor $b\in A$ such that for each integer $n\ge 1$ and element $a\in\cande nA$, one has
\[
ba^{3}I^{n-1}\subseteq \cande {n+2}{\End_{A}(M)}\,.
\]
\end{theorem}

\begin{proof}
Set $\Lambda:=\End_{A}(M)$ and let $r$ be the rank of the $A$-module $M$. There is then an exact sequence of finitely generated $A$-modules
\[
0 \lra A^{r} \lra M \lra T \lra 0
\]
and a non-zerodivisor $b$ in $A$ with $b T=0$. Applying $\Hom_A(M,-)$ to it yields a $\Lambda$-module $D$ with $bD=0$ and $\Hom_A(M,A)^{r}=\syz_\Lambda D$. Since $r\ge 1$, it follows that
\begin{equation}
\label{eq:ascent}
b\Ext_\Lambda^{\ges 1}(\Hom_A(M,A),\mod\Lambda)=0\,.
\end{equation}
Fix an $N\in\mod A$ and set $L=\syz^{n-1}_{A}N$. For any $a\in\cande nA$, from Remark~\ref{re:ext2} one gets an exact sequence of $A$-modules.
\[
0 \lra (0\colon_L\,a) \lra L\oplus\syz_{A}L \lra A^{s} \lra L/aL \lra 0
\]
Applying $\Hom_A(M,-)$ then induces exact sequences of $\Lambda$-modules
\begin{gather*}
0 \lra \Hom_A(M,\syz_{A}(L/aL)) \lra \Hom_A(M,A)^{s} \lra E_3 \lra 0\,,\\
E_1 \lra \Hom_A(M,L)\oplus\Hom_A(M,\syz_{A} L) \lra \Hom_A(M,\syz_{A}(L/aL)) \lra E_2\,,
\end{gather*}
where $a$ annihilates the $E_i$. It follows from the first exact sequence and \eqref{eq:ascent} that 
\[
ab\Ext_\Lambda^{\ges 1}(\Hom_A(M,\syz_{A}(L/aL)),\mod\Lambda)=0\,.
\]
Then the second exact sequence and Remark~\ref{re:4terms} yield
\[
a^3b\Ext_\Lambda^{\ges 1}(\Hom_A(M,L),\mod\Lambda)=0\,.
\]
For each $j\ge0$ there is an exact sequence $0 \to \syz^{j+1}_{A}N \to A^{t_j} \to \syz^j_{A}N \to 0$ of $A$-modules, which induces an exact sequence of $\Lambda$-modules
\[
0 \lra \Hom_A(M,\syz^{j+1}_{A}N) \lra \Hom_A(M,A)^{t_j} \lra \Hom_A(M,\syz^j_{A}N) \lra \Ext_A^{1}(M,\syz^{j+1}_{A}N)\,.
\]
Using these exact sequences, \eqref{eq:ascent} and Remark~\ref{re:4terms} a descending induction on $j$ yields
\[
(a^3b)(bI)^{n-j-1}\Ext_\Lambda^{\ges(n-j)}(\Hom_A(M,\syz^j_{A}N),\mod\Lambda)=0\quad\text{for $0\le j\le n-1$.}
\]
Therefore $a^{3}b^{n}I^{n-1}$ annihilates $\Ext_\Lambda^{\ges n}(\Hom_A(M,N),\mod\Lambda)$, and so, by Lemma~\ref{le:2ndsyzygy}, also $\Ext_\Lambda^{\ges (n+2)}(\mod\Lambda,\mod\Lambda)$. This completes the proof of the theorem.
\end{proof}

\subsection*{Non-zerodivisors}
For applications, the ``useful'' cohomology annihilators are ones that are also non-zerodivisors. The development below is driven by this consideration.

\begin{lemma}
\label{le:rank0}
Let $A$ be a commutative noetherian ring. If $M$ is a finitely generated $A$-module with rank, then $\ann_{A}\Ext^{1}_{A}(M,\syz_{A}M)$ contains a non-zerodivisor.
\end{lemma}

\begin{proof}
For any associated prime $\fp\in\Spec A$, the $A_{\fp}$-module $M_{\fp}$ is free, so 
\[
\Ext^{1}_{A}(M,\syz_{A}M)_{\fp}\cong \Ext^{1}_{A_{\fp}}(M_{\fp},(\syz_{A}M)_{\fp})=0
\]
This means that the annihilator ideal in question is not contained in such $\fp$, as desired.
\end{proof}

\begin{lemma}
\label{le:rank}
Let $A$ be a commutative noetherian ring and $\Lambda$ be noether $A$-algebra of positive rank. The following statements hold.
\begin{enumerate}[\rm(1)]
\item
The natural map $A\to\Lambda$ is injective.
\item
If $A$ is reduced and an ideal $J\subseteq \cent \Lambda$ contains a non-zerodivisor, then so does  $J\cap A\subseteq A$.
\item
If $\can\Lambda$ contains a non-zerodivisor, then $A$ is reduced.
\end{enumerate}
\end{lemma}

\begin{proof}

(1) Since $\Lambda$ has rank as an $A$-module, there is an injective homomorphism $A^{r}\to\Lambda$ of $A$-modules,  for $r:=\rank_A\Lambda$.  If $a$ is in the kernel of the ring homomorphism $A\to\Lambda$, then $a\Lambda=0$, and hence $aA^r=0$. Since $r>0$, it follows that $a=0$.

(2) Set $R:=\cent\Lambda$. Suppose that the ideal $I:=J\cap A$ of $A$ only contains zerodivisors. It is then contained in an associated prime $\fp$ of $A$. The extension $A/I\hookrightarrow R/J$ is module-finite, so there exists a prime ideal $\fq$ of $R$ containing $J$ such that $\fq\cap A=\fp$. Since $A$ is reduced, $\fp$ is a minimal prime so it follows that $\fq$ is a minimal prime of $R$; see \cite{Matsumura80}*{Theorem 5 ii)}. Hence $\fq$ consists of zerodivisors of $R$, contradicting the fact that $J$ contains a non-zerodivisor.

(3) It suffices to prove that for any $\fp\in\Spec A$, if $\depth A_{\fp}$ is zero, then $A_{\fp}$ is a field. Since ${\can \Lambda}_{\fp}$ is contained in $\can{\Lambda_{\fp}}$, by Lemma~\ref{le:ca-localize}(1), we may replace $A$ and $\Lambda$ by their localizations at $\fp$, and assume that $A$ is a local ring, say with maximal ideal $\fm$, with $\depth A=0$, and verify that  $A$ is a field. Now $\Lambda$
is a finitely generated free $A$-module.

By assumption, there exist a non-zerodivisor $x\in \cent\Lambda$ and an integer $n\ge 1$ such that $x\Ext_{\Lambda}^{n}(\mod\Lambda,\mod\Lambda)=0$. Since $A$ is a local ring of depth $0$, it suffices to verify that $A$ is regular, and this follows from the

\begin{claim}
$\Ext^{n}_{A}(k,k)=0$ for $k=A/\fm$, the residue field of $A$.
\end{claim}

Set $t:=\rank_k\Ext_{A}^n(k,k)$ and assume $t\ge 1$. As $\Lambda$ is a free $A$-module one has
\[
\Ext_{\Lambda}^n(k\otimes_A\Lambda,k\otimes_A\Lambda)\cong\Ext_A^n(k,k)\otimes_A\Lambda\cong(k\otimes_A\Lambda)^{t}\cong(\Lambda/\fm\Lambda)^{t}
\]
as $\cent\Lambda$-modules. Thus $x(\Lambda/\fm\Lambda)^{t}=0$, that is to say, $x$ belongs to $\fm\Lambda$. Since $A$ has depth zero, it possesses a nonzero socle element, say $a$. Hence $ax\in a(\fm\Lambda)=(a\fm)\Lambda=0$. This contradicts the fact that $x$ is a non-zerodivisor of $\cent\Lambda$ (recall from (1) that the map $A\to\cent\Lambda$ is injective) and justifies the claim. 
\end{proof}

\begin{corollary}
\label{co:ascent-descent-nzd}
Let $A$ be a commutative noetherian ring.
\begin{enumerate}[{\quad\rm(1)}]
\item If $\Lambda$ is a noether $A$-algebra of positive rank and $\cande n\Lambda$ contains a non-zerodivisor, then $A$ is reduced, and $\cande nA$ contains a non-zerodivisor.
\item
If  $\cande nA$ contains a non-zerodivisor, then for any torsion-free $A$-module $M$ of positive rank,  $\cande n{\End_{A}(M)}$ contains a non-zerodivisor.
\end{enumerate}
\end{corollary}

\begin{proof}
(1) Given such a $\Lambda$, Lemma~\ref{le:rank0} implies that the ideal $I:=\ann_{A}\Ext^{1}_{A}(\Lambda,\syz_{A}\Lambda)$ of $A$ contains a non-zerodivisor, and it follows from Lemma~\ref{le:rank}(3) that $A$ is reduced. Then Lemma~\ref{le:rank}(2) yields that the ideal $A\cap \cande n{\Lambda}$ of $A$ contains a non-zerodivisor as well. Products of non-zerodivisors remain non-zerodivisors so it remains to apply Theorem~\ref{th:descent}.

(2) Set $\Lambda:=\End_{A}(M)$ and $I:=\ann_{A}\Ext^{1}_{A}(M,\syz_{A}M)$. By  Lemma~\ref{le:rank0}, the ideal $I$ of $A$ contains a non-zerodivisor, so it follows from Theorem~\ref{th:ascent} that the annihilator ideal of $\Ext^{n}_{\Lambda}(\mod\Lambda,\mod\Lambda)$, viewed 
as an $A$-module, contains a non-zerodivisor as well. It remains to note that, since the $A$-module $M$ is torsion-free, the map $A\to \Lambda$ takes non-zerodivisors of $A$ to non-zerodivisors of $\cent\Lambda$.
\end{proof}

\begin{corollary}
\label{co:descent}
Let $A$ be a reduced ring. If there exists a noether $A$-algebra $\Lambda$ of positive rank and a finitely generated $A$-module $G$ such that $\mod\Lambda \subseteq \thick_A^{n}(G)$ for some $n\ge 1$ (for example, if $\gldim\Lambda\le n$), then $\cande {n+1}A$ contains a non-zerodivisor.
\end{corollary}

\begin{proof}
There is a non-zerodivisor $x\in A$ with $x\Ext_A^{\ges 1}(G,\mod A)=0$ by Lemma \ref{le:ca-module}(1). It then follows that  $x^{n}\Ext_{A}^{\ges n}(\mod\Lambda,\mod A)=0$. The assertion follows from \eqref{eq:descent2}.
\end{proof}

The preceding result applies when $A$ is a subring of a (commutative) regular ring $S$ of finite Krull dimension such that $S$ is a finitely generated $A$-module. Quotient singularities provide one such family of examples, and this seems worth recording.

\begin{corollary}
\label{co:quotient-singularities}
Let $k$ be a field and $S$ either $k[\ulx ]$, the polynomial ring, or  $k\llbracket \ulx\rrbracket$, the formal power series ring, in commuting indeterminates  $\ulx = x_{1},\dots, x_{d}$ over $k$. If $G$ is a finite subgroup of the $k$-algebra automorphisms of $S$, then $\cande{d+1}{S^{G}}$ is non-zero. 
\end{corollary}

\begin{proof}

The only thing to note is that $S$ is a finitely generated module over $S^{G}$ (see, for example, \cite{SchejaStorch73}*{Propositions (19.3) and (19.4)}) and that the global dimension of $S$ equals $d$. Thus, Corollary~\ref{co:descent} applies.
\end{proof}

\begin{remark}
\label{re:ncr}

Following \cite{DaoIyamaTakahashiVial12}, we say that $A$ admits a \emph{noncommutative resolution} if there exists a finitely generated faithful $A$-module $M$ such that $\End_A(M)$ has finite global dimension. Corollary~\ref{co:descent} thus implies that when $A$ is a domain admitting a noncommutative resolution, $\can A$ contains a non-zerodivisor. Since a finitely generated faithful module $M$ over a domain $A$ has positive rank, so does $\End_A(M)$ as an $A$-module. 
\end{remark}

\begin{remark}
The hypothesis in Corollary~\ref{co:ascent-descent-nzd} that $\Lambda$ is module-finite over $A$ is necessary: If  $A$ is a non-reduced ring possessing a prime ideal $\fp$ such that $A_\fp$ is a field (for example,. $k[[x,y]]/(x^2y)$), then $\can{A_{\fp}}$ contains a non-zerodivisor.
The condition that $\Lambda$  has positive rank over $A$ is also needed: For any $A$ and maximal ideal $\fm$ of $A$, the ideal $\can{A/\fm}$ contains a non-zerodivisor.
\end{remark}

\subsection*{Commutative rings}
Here is one application of Theorems \ref{th:ascent} and \ref{th:descent}. Note that any Nagata ring (hence any excellent ring) satisfies its hypothesis, by definition; see \cite{Matsumura80}*{\S31}.  Thus the result below, whose statement was suggested to us by Ken-ichi Yoshida, subsumes ~\cite{Wang94}*{Proposition 2.1} that deals with one-dimensional reduced complete local rings.

\begin{corollary}
\label{co:yoshida}
Let $R$ be a commutative noetherian ring such that its integral closure $\overline{R}$ is finitely generated as an $R$-module.
Then the ideal $\cande n R\subseteq R$ contains a non-zerodivisor if and only if so does $\cande n{\overline{R}}\subseteq \overline{R}$.
\end{corollary}

\begin{proof}
We have $R\subseteq\overline{R}\subseteq\QQ(R)$, where $\QQ(R)$ denotes the total ring of fractions of $R$. As $\overline{R}$ is module-finite over $R$, there is a non-zerodivisor $a$ of $R$ with $a\overline{R}\subseteq R$. Hence $\rank_R(\overline{R}/R)=0$, and $\rank_R(\overline{R})=1$.  The `if' part now follows from Corollary~\ref{co:ascent-descent-nzd}(1).

Now consider maps $\phi\colon \overline{R}\to\End_R(\overline{R})$ and $\psi\colon \End_R(\overline{R})\to\overline{R}$ by $\phi(x)(y)=xy$ for $x,y\in\overline{R}$ and $\psi(f)=f(1)$. It is easy to verify that these are mutually inverse bijections. Therefore $\overline{R}\cong\End_R(\overline{R})$. Since $\overline{R}$ is torsionfree as an $R$-module, Corollary~\ref{co:ascent-descent-nzd}(2) completes the proof of the converse.
\end{proof}

The preceding result can be used to give another proof of Corollary~\ref{co:nagata1}, though the crux of argument is essentially the same.

\section{Strong generators for derived categories}
\label{se:dcat}
Let $\Lambda$ be a noetherian ring and $\dbcat(\mod \Lambda)$ the bounded derived category of $\mod\Lambda$.  In analogy with the construction of $\thick^{n}_{\Lambda}(C)$ in Section~\ref{se:generators}, for any complex $C$ in $\dbcat(\mod\Lambda)$ one can define a tower of subcategories of $\dbcat(\mod\Lambda)$:
\[
\{0\}  \subseteq \thick^{1}_{\sfD}(C) \subseteq \cdots \subseteq \thick^{n}_{\sfD}(C) \subseteq  \thick^{n+1}_{\sfD}(C)\subseteq
\]
The only difference is that one is allowed the use of $\{\susp^{i}C\}_{i\in\bbZ}$, the suspensions of $C$, in building $\thick^{1}_{\sfD}(C)$, and that exact sequences are replaced by exact triangles. For details, see, for example, \cite{BondalVandenBergh03}. As in \cite{BondalVandenBergh03}, see also \cite{Rouquier08a}, we say that $\dbcat(\mod\Lambda)$ is \emph{strongly finitely generated}  if there exists a $C$ and an integer $n$ such that $\thick^{n}_{\sfD}(C)=\dbcat(\mod \Lambda)$. Such a $C$ is then a \emph{strong generator} for $\dbcat(\mod \Lambda)$.

In the same vein, one can construct a tower of subcategories in $\{\thick^{n}_{\dsing}(C)\}_{n\ges 0}$ in $\dsing(\mod\Lambda)$, the \emph{singularity category} of $R$, introduced by Buchweitz~\cite{Buchweitz87} under the name `stable derived category'.
This then leads to a notion of a strong generator for this category.

The gist of the result below, which is well-known, is that any strong generator for $\mod \Lambda$, in the sense of Corollary~\ref{co:strong-generation}, is also a strong generator for $\dbcat(\mod\Lambda)$ and for $\dsing(\mod\Lambda)$.

\begin{lemma}
\label{le:dcat}
Let $\Lambda$ be a noetherian ring and $G$ a finitely generated $\Lambda$-module.
\begin{enumerate}[{\quad\rm(1)}]
\item
If $|\syz^{s}(\mod\Lambda)|\subseteq \sad{G}_{n}$ for  integers $s,n\ge 1$, then $\dsing(\mod\Lambda)= \thick^{n+1}_{\dsing}(G)$.
 \item
If $\mod\Lambda=\thick^{n}_{\Lambda}(G)$ for an integer $n\ge 1$, then $\dbcat(\mod\Lambda)=\thick^{n+1}_{\sfD(\Lambda)}(G)$.
\end{enumerate}
\end{lemma}

\begin{proof}
We view $\mod\Lambda$ as a subcategory of $\dbcat(\mod\Lambda)$, and also of $\dsing(\Lambda)$, as usual. 

(1) Since any projective module is zero in $\dsing(\Lambda)$, for any finitely generated $\Lambda$-module $M$ and integer $i$, there is an isomorphism $\syz^{i}M\cong \susp^{-i}M$ in $\dsing(\Lambda)$. Given this, it is easy to verify that $\sad{G}_{n}\subseteq\thick^{n}_{\dsing}(G)$ for each integer $n$, and then (1) follows.

(2) It is clear from the constructions that $\thick_{\Lambda}^{n}(G) \subseteq\thick^{n}_{\sfD}(G)$ for each integer $n$.  Fix a $C$ in $\dbcat(\mod\Lambda)$, and let $Z$ and $B$ denote the cycles and the boundaries of $C$, respectively. As $Z$ and $B$ are graded $\Lambda$-modules, they are in $\thick^{t}_{\sfD}(G)$. Viewing them as complexes, with zero differential, there is an exact sequence of complexes of $\Lambda$-modules
\[
0\lra Z \lra C \lra \susp B \lra 0\,.
\]
It thus follows that $C$ is in $\thick^{t+1}_{\sfD(\Lambda)}(G)$, as desired.
\end{proof}

In view of Corollary~\ref{co:strong-generation}(1) and Lemma~\ref{le:dcat}(2), the following result is a direct consequence of Theorems~\ref{th:complete-equicharacteristic} and \ref{th:eft}.

\begin{corollary}
\label{co:dcat-generators}
Let $R$ be a commutative ring. If $R$ is an equicharacteristic excellent local ring or essentially of finite type over a field, then $\dbcat(R)$ is strongly finitely generated. \qed
\end{corollary}

The part of this result dealing with rings essentially of finite type over a field was proved by Rouquier~\cite{Rouquier08a}*{Theorem~7.38} when the field is perfect, and extended to the general case by Keller and Van den Bergh~\cite{KellerVandenBergh08}*{Proposition~5.1.2}. For complete local rings containing a perfect field, the result was proved by Aihara and Takahashi~\cite{AiharaTakahashi11}*{Main Theorem}.

\begin{remark}
\label{re:dcat-not-generated}
Example~\ref{ex:non-open-loci} furnishes  rings (even commutative noetherian) whose module categories do not have strong generators, since the singular loci of the rings there are not closed; see Theorem~\ref{th:generator-sing}. For other examples, see \cite{DaoTakahashi11}*{Theorem~4.4}.
\end{remark}

Osamu Iyama asked us if it would be possible to bound the dimension of $\dbcat(R)$ in terms of the least integer $n$ such that $\cande nR$ contains a non-zerodivisor. We do not know the answer, but offer the following examples. We thank Kei-ichiro Iima for suggesting the second one.

\begin{example}
Let $k$ be a perfect field of characteristic $\ge 3$, and set
\[
 R:=k[[x,y,z,w]]/(x^2-y^2,y^2-z^2,z^2-w^2)\quad \text{and}\quad J:=(xyz,xyw,xzw,yzw)\subset R\,.
 \]
Then $R$ is a one-dimensional reduced complete intersection and $J$ is its Jacobian ideal. The element $xyz$, which is a non-zerodivisor on $R$, is contained in $\cande 2R$, by \cite{Wang94}*{Theorem 5.3}. However there are inequalities
\[
\dim \dbcat(R)\ge \dim \dsing(R) \ge {\mathrm{codim}}\, R-1=2\,,
\]
where the first one is clear and the second one follows from \cite{BerghIyengarKrauseOppermann10}*{Corollary~5.10}.

Here is a example that is also a domain: Let $k$ be a field and let $R$ the numerical semigroup ring $k[[t^{16},t^{17},t^{18},t^{20},t^{24}]]$, viewed as a subring of $k[[t]]$. Then $R$ is isomorphic to $k[[x,y,z,w,v]]/(y^2-xz,z^2-xw,w^2-xv,v^2-x^3)$, so a complete intersection domain. By \cite{Wang94}*{Proposition~3.1} we have $t^{64}R\subseteq\cande2A$, while 
\[
\dim\dbcat(R)\ge\dim\dsing(R)\ge\operatorname{codim}R-1=3\,.
\]
\end{example}

\subsection*{A decomposition of the derived category}
As is apparent from what has been discussed so far, the existence of cohomology annihilators of a ring has an impact on structure of its module category, and hence on its derived category. There is even a clear and direct connection between the two: When $\Lambda$ is noether algebra, $\cande nR$ contains a non-zerodivisor $a$ if and only if $\dbcat(\Lambda)$ is a product of the subcategory consisting of 
complexes annihilated by $a$ and $\thick_{\sfD}^{n}(\Lambda)$. A proof of this result will be presented in \cite{IyengarTakahashi14b}.

\end{document}